\documentclass[journal]{IEEEtran}

\usepackage{tikz}
\usetikzlibrary{patterns}
\usepackage{pgfplots}
\pgfplotsset{compat=1.15}
\usepackage{graphicx}

\usepgfplotslibrary{fillbetween}

\usepackage{amssymb}
\usepackage{amsmath}
\usepackage{amsthm}
\usepackage{array}
\usepackage{bm}
\usepackage{dsfont}

\theoremstyle{definition}

\newtheorem{theorem}{Theorem}
\newtheorem{lemma}[theorem]{Lemma}
\newtheorem{remark}{Remark}

\usepackage{algorithm}
\usepackage{algorithmic} 
\usepackage{multirow}

\usepackage{eurosym}
\usepackage{xcolor,colortbl}

\usepackage{pifont}

\newcommand\numberthis{\addtocounter{equation}{1}\tag{\theequation}}

\newcommand{\norm}[1]{\left\lVert#1\right\rVert}
\usepackage{bbm}

\usepackage{subfigure}
\usepackage{pgfkeys}
    \def\addlegendimage{\csname pgfplots@addlegendimage\endcsname}
   
\begin{document}


\title{Embedding Dependencies Between Wind Farms in Distributionally Robust Optimal Power Flow}

\author{Adriano~Arrigo,
        Jalal~Kazempour,~\IEEEmembership{Senior~Member,~IEEE,}
        Zacharie~De~Grève,~\IEEEmembership{Member,~IEEE,}
        Jean-François~Toubeau,~\IEEEmembership{Member,~IEEE,}
        and~François~Vallée,~\IEEEmembership{Member,~IEEE} \vspace{-0.6cm}
\thanks{A. Arrigo, Z. De Grève and F. Vallée are with the Power Systems and Markets Research group at University of Mons, 31, Boulevard Dolez, 7000, Mons, Belgium. E-mails : \{adriano.arrigo, zacharie.degreve, francois.vallee\}@umons.ac.be.}
\thanks{J.-F. Toubeau is with the TME Branch, Energy Institute, University of Leuven (KU Leuven), 13, Oude Markt, 3000, Leuven, Belgium. E-mail: jean-francois.toubeau@kuleuven.be.}%
\thanks{J. Kazempour is with the Department of Wind and Energy Systems at Technical University of Denmark, 325, Elektrovej, 2800 Kgs. Lyngby, Denmark. E-mail: jalal@dtu.dk.}}




\maketitle

\begin{abstract}
The increasing share of renewables in the electricity generation mix comes along with an increasing uncertainty in power supply. In the recent years, distributionally robust optimization has gained significant interest due to its ability to make informed decisions under uncertainty, which are robust to misrepresentations of the distributional information (e.g., from probabilistic forecasts). This is achieved by introducing an ambiguity set that describes the potential deviations from an empirical distribution of all uncertain parameters. However, this set typically overlooks the inherent dependencies of uncertainty, e.g., spatial dependencies of weather-dependent energy sources. This paper goes beyond the state-of-the-art models by embedding such dependencies within the definition of ambiguity set. In particular, we propose a new copula-based ambiguity set which is tailored to capture any type of dependencies. The resulting problem is reformulated as a conic program which is kept generic such that it can be applied to any decision-making problem under uncertainty in power systems. \color{black} Given the Optimal Power Flow (OPF) problem as one of the main potential applications, we illustrate the performance of our proposed distributionally robust model applied to \textit{i)} a DC-OPF problem for a meshed transmission system and \textit{ii)} an AC-OPF problem using \textit{LinDistFlow} approximation for a radial distribution system. \color{black} 
\end{abstract}


\IEEEpeerreviewmaketitle



\vspace{-0.3cm}
\section{Introduction}
\label{Sect:Introduction}



\IEEEPARstart{T}{he} 
continuous balance between power injections and offtakes is paramount to ensure the safe operation of power systems, however it is being challenged by the   uncertainty introduced by the weather-dependent renewable power generation \cite{NERC}. To cope with such an uncertainty, it is crucial to develop uncertainty-aware decision-making tools for operational and planning purposes. While the power system operators traditionally use a deterministic optimization approach by considering a single-point forecast of renewable power generation, stochastic approaches account for a probabilistic forecast, allowing to make more informed operational and planning decisions (depending on the forecast accuracy) \cite{Morales2013Integrating}. 

The power system research community has put a lot of efforts in the last decades for developing various types of stochastic decision-making tools, ranging from 
scenario-based stochastic programming \cite{juanmi} to robust optimization \cite{bertsimas2} and chance-constrained programming \cite{Lubin}. However, these techniques disregard the inherent misrepresentation of the input probabilistic information, which may lead to high discrepancies between the expected results and the actual (ex-post) realizations \cite{Christos}. To address this issue, we focus on Distributionally Robust Optimization (DRO) \cite{Kuhn2}, which has shown the capability to outperform all previous techniques in presence of high uncertainties.

DRO allows to hedge against the inherent errors arising from modeling the probabilistic distribution that is typically derived from the probabilistic forecast. These errors may propagate to the resulting decisions and compromise the reliability and cost-efficiency of the power system. The exposure to uncertainty in the probability distribution itself is called \textit{ambiguity} \cite{Keynes1921}. By design, DRO considers a family of potential distributions, the so-called ambiguity set, to hedge against any inexactness or biasedness of the distribution function. Owing to its appealing properties, one popular approach to define an ambiguity set is based on the Wasserstein probability metric, which calculates the distance between two distribution functions \cite{Kuhn}. In the Wasserstein DRO approach, all distributions in the neighborhood of a central empirical distribution, based on historically observed samples of uncertainty, are collected and fed to the decision-making problem. Although this framework accounts for potential forecast errors, there is no guarantee that the distributions within the ambiguity set remain realistic (e.g., from a correlation perspective). As our main contribution, we go beyond the state of the art by developing a DRO approach for power systems that is aware of all potential \textit{dependencies} among uncertain parameters, especially the spatial dependencies of weather-dependent renewable energy sources. 



In the current literature, there is a willingness to enhance the Wasserstein ambiguity set by removing the unrealistic distributions of renewable power generation uncertainty. To do so, supplementary constraints should be incorporated into the design of the ambiguity set. References \cite{Kuhn} and \cite{Arrigo2020Wasserstein} include the support information in the definition of Wasserstein ambiguity set, excluding distributions with unrealistic realizations of uncertainty. An example of such unrealistic distributions is those with a negative renewable power generation. References \cite{li2018value} and \cite{Esteban2019Data} include modality information, i.e., the number of spikes in a probability distribution, within the ambiguity set to get rid of potentially unrealistic distributions, e.g., those with two or more spikes. References \cite{Wang2018Risk} and \cite{powertech} embed information on dependencies, which is based on imposing the value of a covariance matrix in the ambiguity set definition. \color{black} However, the covariance matrix is only able to capture \textit{linear} relations and fails to capture more complex non-linear dependencies (e.g., stemming from wind power generation). Moreover, the models derived in these works usually impose the value of covariance via a positive semi-definite matrix restriction which is not flexible, i.e., the distributions within the set must follow the dependence structure contained in an arbitrarily fixed semi-definite cone centered on the empirically observed covariance matrix. \color{black} Overall, all these works allow to improve the representation of multi-dimensional uncertainties, however all have their own limitations and cannot represent the whole dependence structure.

%
%
%
Several additional works have contributed towards embedding the full non-linear dependence information into the distributionally robust optimization. However, it is worth emphasizing that there are currently very few contributions in this research strand, especially focusing on power system applications. References \cite{Gao2017DataDrivenRO}-\cite{Pflug2018Review} envisage the use of \textit{copula} as the mathematical object to describe the full dependence structure among random variables. A copula is a distribution function with uniformly distributed marginals that contains solely information about the dependence structure among uncertain parameters \cite{Nelsen1999Introduction}-\cite{Sklar1959}. To incorporate the dependence structure into the optimization problem, \cite{Gao2017DataDrivenRO}-\cite{Pflug2018Review} introduce a constraint binding the Wasserstein distance between the copula of the distributions inside the ambiguity set and the one of the empirical distribution. In all these works, the marginals are fixed and assumed to be known. In other words, the only source of uncertainty stems from the dependence structure, and the ambiguity around the empirical distribution is no longer accounted for. This uncertainty scheme is usually suited for portfolio management problems in finance applications, but is expected to be insufficient for power system applications, where the marginal distributions of each renewable energy source are usually not predicted with high accuracy. 


To address this issue, \cite{Gao2017Distributionally} represents the first effort in the operations research literature towards embedding the full dependence information along with the ambiguity around the empirical distribution function. To do so, the authors combine the findings in \cite{Gao2017DataDrivenRO}-\cite{Pflug2018Review} with the classical approach to DRO \cite{Kuhn}. They develop a Wasserstein ambiguity set that uses copula as the mathematical object to describe the full dependence structure among renewable energy sources. 

In this paper, inspired by \cite{Gao2017Distributionally}, we propose a copula-based ambiguity set as the framework that \textit{generalizes} \textit{(i)} the constraint on the second-order moment \cite{Wang2018Risk}-\cite{powertech}, and, \textit{(ii)} the model with known marginals \cite{Gao2017DataDrivenRO}-\cite{Pflug2018Review} by enabling the DRO problem to consider \textit{any} potential type of correlation (not necessarily linear) among uncertain parameters. By this generalization, we are able to properly model the renewable power generation uncertainty that may exhibit any shape of dependencies \cite{Vallee2007Impact}. To the best of our knowledge, this is the first effort in the literature to introduce copula-based ambiguity sets for power system applications. In particular, our contributions are threefold: \color{black}

\begin{itemize}
\item[\textit{(i)}] We first develop an optimization program in Lemma \ref{Lemma} which computes the value of an empirical cumulative distribution function. This allows us to link any distribution within the ambiguity set to its corresponding copula.
\item[\textit{(ii)}] Building up on Lemma 1, we develop a conic reformulation of a distributionally robust worst-case expectation problem\footnote{The worst-case expectation problem is a well-known problem in the field of DRO \cite{Guo2019Data}. Its generic reformulation usually facilitates the application of DRO to any kind of problem where it naturally appears.} in Theorem 1, using the proposed copula-based ambiguity set. Theorem 1 and the solution approach in Appendix C are kept generic such that they could be applied to a broad range of problems in power systems. \color{black}
\item[\textit{(iii)}] We apply the proposed conic formulation to a day-ahead distributionally robust Optimal Power Flow (OPF) problem with dependent renewable power generation uncertainty. We explore both DC and AC optimal power flow formulations and numerically show the benefits of the proposed reformulation in both cases compared to a state-of-the-art distributionally robust model, while identifying potential paths for future research. \color{black}
\end{itemize}

The remainder of this paper is structured as follows. Sections \ref{Sect:TradMetric} and \ref{Sect:Uncertainty} provide a definition for the traditional metric-based and the proposed copula-based ambiguity sets, respectively. Section \ref{Sect:Reformulation} explains the reformulation process of a generic worst-case expectation problem under the copula-based ambiguity set. Section \ref{Sect:ModelFormulation} introduces the day-ahead distributionally robust OPF problem and details the required reformulations for both DC and AC formulations. Section \ref{Sect:NumericalStudy} discusses the proposed model through an extensive numerical analysis. Main conclusion and prospects are given in Section \ref{Sect:Conclusion}. Finally, the mathematical proofs of Lemma 1 and Theorem 1 are provided in Appendixes A and B, respectively, while the proposed solution approach for Theorem 1 is discussed in Appendix C.



\color{black}
\section{The Traditional Metric-Based Ambiguity Set}
\label{Sect:TradMetric}

We start with the definition of an empirical distribution function. Consider $N$ number of equiprobable historical observations of the uncertainty, i.e., $\widehat{\xi}_i$, where $i=\{1, 2, ..., N\}$. The empirical distribution  $\widehat{\mathbb{Q}}_N$ is defined as
\begin{equation}
\widehat{\mathbb{Q}}_N= \frac{1}{N} \sum\limits_{i = 1}^N \delta_{\widehat{\xi}_i},
\end{equation}
where $\delta_{\widehat{\xi}_i}$ represents the Dirac distribution centered on $\widehat{\xi}_i$ and is assigned with a probability equal to $\frac{1}{N}$. Note that every symbol with a hat, e.g., $\widehat{\mathbb{Q}}_N$ and $\widehat{\xi}_i$, corresponds to historical observations. Next, we define the Wasserstein probability metric. The metric $ d_W \left( \mathbb{Q} , \widehat{\mathbb{Q}}_N \right) : \mathcal{A} \times \mathcal{A} \rightarrow \mathbb{R}^+$ computes the distributional distance between two distribution functions $\mathbb{Q}$ and $\widehat{\mathbb{Q}}_N$ as the optimal value of a transportation problem between the probability mass of those two distributions \cite{WassersteinDef}. The complete definition of this metric is given in the online appendix \cite{OnlineAppendix}. Note that $\mathcal{A}$ defines the space of all distribution functions.

Given the definition of the Wasserstein probability metric, the metric-based  ambiguity set $\mathcal{M}_1$ is defined as
\begin{equation}\label{M1}
\mathcal{M}_1 = \left\lbrace \mathbb{Q} \in \mathcal{A} \, \middle\vert \, d_W \left( \mathbb{Q} , \widehat{\mathbb{Q}}_N \right) \leq \theta_1 \right\rbrace,
\end{equation}

\noindent which contains a family of distributions $\mathbb{Q}$ in the space of all distribution functions $\mathcal{A}$ that are in the neighbourhood of the empirical distribution function $\widehat{\mathbb{Q}}_N$. 
The distributional distance $d_W \left( \mathbb{Q} , \widehat{\mathbb{Q}}_N \right)$ is limited to be lower than  or equal to $\theta_1 \in \mathbb{R}^+$. This value is tuned by the decision-maker, and is referred to as the radius of the ambiguity set. One can intuitively interpret that a larger value of $\theta_1$ yields an ambiguity set that contains more distributions, implying that the decision-maker is less certain about the true distribution of the uncertainty. The reader is referred to \cite{Kuhn} for more details about the traditional Wasserstein ambiguty set. In Section \ref{Sect:Uncertainty}, we innovately enhance the traditional Wasserstein metric-based ambiguity set with a copula model to prevent unrealistic uncertainty distributions within the ambiguity set.\color{black}

\vspace{-0.3cm}
\section{The Proposed Copula-Based Ambiguity Set}
\label{Sect:Uncertainty}

We first mathematically define copula $\mathbb{C}$ of the distribution $\mathbb{Q}$. Let us consider $\vert \mathcal{W} \vert$ number of renewable power units, e.g., wind farms. Let $\widetilde{\xi} \in \mathbb{R}^{\vert \mathcal{W} \vert}$ be linked to distribution $\mathbb{Q}$. Unlike symbols with a hat that refer to historical observations, those with a tilde, e.g., $\widetilde{\xi}$, correspond to  uncertain parameters.  The symbol $\widetilde{\xi}_k$ refers to the uncertain  generation of the renewable power unit $k$. Mathematically speaking, the copula $\mathbb{C}$ of distribution $\mathbb{Q}$ is defined as the cumulative distribution function of the uncertain parameter $\widetilde{U}$ \cite{Nelsen1999Introduction}-\cite{Sklar1959}, i.e., 
\begin{align}
& \left( \widetilde{U}_1, ..., \widetilde{U}_k, ..., \widetilde{U}_{\vert \mathcal{W} \vert} \right) \nonumber \\ 
& \hphantom{\text{-----------}} = \Bigg(F_1 \left( \widetilde{\xi}_1 \right), ..., F_k \left( \widetilde{\xi}_k \right), ..., F_{\vert \mathcal{W} \vert} \left( \widetilde{\xi}_{\vert \mathcal{W} \vert} \right) \Bigg),  \label{CopulaEq}
\end{align}


\noindent \color{black} where $\widetilde{U} \in \mathbb{R}^{\vert \mathcal{W} \vert}$ is linked to distribution $\mathbb{C}$. In addition, the function $F_k \left( . \right) = \mathbb{Q}_k \left( \tilde{\xi}_k \leq . \right)$ represents the cumulative distribution function of element $\widetilde{\xi}_k$ of the $\vert \mathcal{W} \vert$-dimensional uncertain vector $\widetilde{\xi}$. Note that the probability operator $\mathbb{Q}_k \left( . \right)$ represents the marginal distribution function of the random variable $\tilde{\xi}_k$. Therefore, $F_k \left( \tilde{\xi}_k \right)$ defines a random variable which is uniformly distributed on the interval $\left[ 0, 1 \right]$. \color{black}

By \eqref{CopulaEq}, the resulting distribution of $\widetilde{U}$ has marginals that are uniformly distributed on the interval $\left[ 0, 1 \right]$ and has the property to embody the dependence between the components of $\widetilde{\xi}$. \color{black} In other words, the information contained in any multivariate distribution function can be splitted into \textit{(i)} a collection of marginal distributions (containing the univariate information) and, \textit{(ii)} a copula (containing the dependence information) from which the univariate information has been filtered.


\color{black} We now elaborate on the need for a copula-based ambiguity set. \color{black} Whereas the metric-based ambiguity set $\mathcal{M}_1$ defined in \eqref{M1} allows the decision-maker to incorporate useful information as much as possible into the distributionally robust program, it may still contain erroneous distributions, i.e., distributions that do not reflect with fidelity the potential outcome of the uncertainty. \color{black} For instance, distributions within set $\mathcal{M}_1$ may exhibit a completely different dependence structure than the one observed empirically. Hence, the decisions may unnecessarily be optimized for an over-conservative and/or non-representative insight of uncertainty, resulting in a higher expected total operational cost. In order to avoid such a situation, we aim to eliminate those distributions from the ambiguity set. In that direction, this paper introduces the copula-based ambiguity set $\mathcal{M}_2$ which generally contains more representative distribution functions, such that \color{black} 
%
%
\begin{align}
& \mathcal{M}_2 = \left\lbrace \mathbb{Q} \in \mathcal{A} \,\middle\vert\, \begin{aligned} & d_W \left( \mathbb{Q} , \widehat{\mathbb{Q}}_N \right) \leq \theta_1 \\
& d_W \left( \mathbb{C} , \widehat{\mathbb{C}}_N \right) \leq \theta_2 \end{aligned} \right\rbrace.
\end{align}

The first constraint in the ambiguity set $\mathcal{M}_2$ is identical to the one in $\mathcal{M}_1$, \color{black} yielding the desirable properties of the classical definition of the ambiguity set. \color{black} The newly added second constraint limits the distributional distance $d_W \left( \mathbb{C} , \widehat{\mathbb{C}}_N \right)$ between the endogenously selected copula $\mathbb{C}$ of the endogenously selected distribution $\mathbb{Q}$ and the empirical copula $\widehat{\mathbb{C}}_N$ of the distribution $\widehat{\mathbb{Q}}_N$. This distance should not be greater than $\theta_2 \in \mathbb{R}^+$. Again, $\theta_2$ is a parameter to be tuned by the decision-maker. By restricting the distributions $\mathbb{Q}$ inside the ambiguity set to have a copula $\mathbb{C}$ in the neighbourhood of the empirical one, the distributions inside the ambiguity set will follow a dependence structure which remains close to the historically observed one\footnote{Another potential methodology relies on fixing the value of the second-order moment of the distributions within the ambiguity set \cite{Wang2018Risk}-\cite{powertech}. However, only the \textit{linear} dependence structure will be captured by such a constraint. On the contrary, the copula-based approach allows to capture \textit{any} kind of dependence structure (not necessarily linear) and therefore offers a more general framework.}. Similar to $\theta_1$, a greater value of $\theta_2$ implies that the decision-maker is less confident about the true dependence structure of the uncertainty, and includes distributions whose dependencies are less similar to those of the empirical one within the ambiguity set.

%
%

\vspace{-0.3cm}
\section{Worst-case Expected problem Under the Copula-Based Ambiguity Set}
\label{Sect:Reformulation}


%
%
By design, DRO aims to determine the \textit{worst-case} distribution within the given ambiguity set, and makes decisions in \textit{expectation} with respect to such a worst-case distribution. This section derives reformulation for a generic distributionally robust worst-case expectation problem, using the copula-based ambiguity set $\mathcal{M}_2$. \color{black} This problem writes as 
\begin{equation}\label{DRWCE}
\min_{x \in \mathcal{X}} \enspace \overbrace{\underset{\mathbb{Q} \in \mathcal{M}_2}{\max} \enspace \mathbb{E}^\mathbb{Q} \left[ a \left( x \right)^\top \widetilde{\xi} + b \left( x \right) \right]}^{\text{Worst-case expected cost}},
\end{equation}
\noindent where $x \in \mathcal{X}$ is the vector of decision variables. The inner maximization operator in \eqref{DRWCE} picks the worst-case distribution $\mathbb{Q}$ in the ambiguity set $\mathcal{M}_2$. The probability distribution of the uncertain parameter $\widetilde{\xi} \in \mathbb{R}^{\vert \mathcal{W} \vert}$ is $\mathbb{Q}$. The objective function \eqref{DRWCE} is linear\footnote{\color{black} For the sake of simplicity, we assume linearity of the objective function. This assumption is aligned with the current practice of electricity markets with linear bids. An extension to a non-linear objective function is straightforward, but requires additional reformulations that are out of the scope of this paper. \color{black}}, comprising of the decision-dependent vector $a(x) \in \mathbb{R}^{\vert  \mathcal{W} \vert}$ and the decision-dependent scalar $b(x) \in \mathbb{R}$. 

One key step in deriving the reformulation of \eqref{DRWCE} is to establish an analytical link between the copula $\mathbb{C}$ and its distribution function $\mathbb{Q}$. This can be achieved by using \eqref{CopulaEq}, in which $\mathbb{Q}$ and $\mathbb{C}$ are linked through the marginal cumulative distribution functions $F_k(.)$.


\begin{remark}
Theoretically, the link between variable copula $\mathbb{C}$ and variable distribution $\mathbb{Q}$ is established via the marginal cumulative distribution functions $F_k^\mathbb{Q} \, \forall k$ of $\mathbb{Q}$. However, in practice, these functions are not straightforwardly accessible, due to their complex variable nature which impedes their endogenous reformulation. Consequently, considering these variable functions within the optimization framework would require research efforts that are beyond the scope of this paper. Therefore, we use the \textit{empirical} marginal cumulative distribution functions $F_k^{\hat{\mathbb{Q}}_N} \, \forall k$ of $\hat{\mathbb{Q}}_N$. By doing so, we are able to derive a tractable reformulation of \eqref{DRWCE}. Our hypothesis is that the approximation made by assessing the endogenous copula $\mathbb{C}$ via the functions $F_k^{\hat{\mathbb{Q}}_N}$ is required (to avoid non-linear formulation of functions $F_k^\mathbb{Q}$) and is valid in practice when $F_k^\mathbb{Q}$ and $F_k^{\hat{\mathbb{Q}}_N}$ are close to each other (e.g., when $\theta_1$ is small enough). This setting is typically suited for a day-ahead probabilistic forecast embedded in OPF problems, as further demonstrated in our numerical analysis in Section \ref{Sect:NumericalStudy}. In the following, functions $F_k^{\hat{\mathbb{Q}}_N}$ are denoted $F_k$ for the ease of notation.
\end{remark}


From now on, we denote by $\eta \in \mathbb{R}$, the value of the argument of the function, i.e., $F_k (\eta)$. Let us further clarify this function by a schematic illustration. Fig. \ref{EmpirFofxi} shows the shape of the function $F_k(\eta)$ for the renewable power unit $k$, given the arbitrarily selected eight equiprobable historical observations $\widehat{\xi}_{ki}, \enspace  i \in \left\lbrace 1, ..., N=8 \right\rbrace$ of $\widetilde{\xi}_k$. In particular, this figure shows how the historical observations $\widehat{\xi}_{ki}$, the variable $\eta$ in the $x$-axis and the function $F_k (\eta)$ in the $y$-axis are linked. Accordingly, the empirical marginal cumulative distribution function for the renewable power unit $k$ writes as
%
%
\begin{figure}
\centering
\includegraphics[scale=0.3]{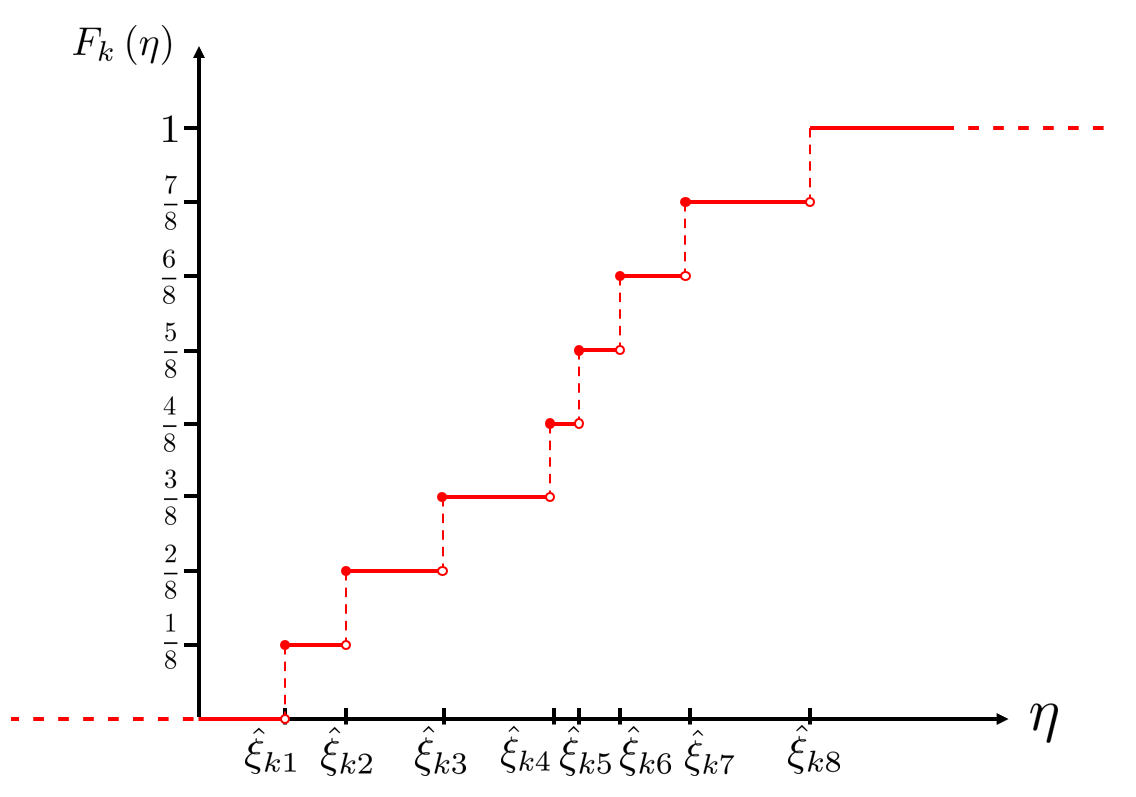}
\caption{\color{black} Illustration of an empirical marginal cumulative distribution function for the renewable power unit $k$  when the number of historical observations is $N$ = 8 (an arbitrarily selected number). \color{black}}
\label{EmpirFofxi}
\end{figure}
\begin{equation}\label{F_of_xi_Lemma}
F_k \left( \eta \right) = \frac{1}{N} \sum\limits_{i = 1}^N \mathbbm{1}_{\eta \geq \widehat{\xi}_{ki}} \, \text{with} \, \mathbbm{1}_{\eta \geq \widehat{\xi}_{ki}} = \begin{cases} 1 \text{ if } \eta \geq \widehat{\xi}_{ki} \\ 0 \text{ otherwise.} \end{cases} 
\end{equation}

\begin{lemma}\label{Lemma}
The empirical marginal cumulative distribution function $F_k \left( \eta \right)$ for the renewable power unit $k$  is equivalent to the following linear optimization program:
\begin{subequations}\label{F_of_xi_LemmaOpti}
\begin{align}
& \max_{z_{ki}} \enspace \frac{1}{N} \sum\limits_{i = 1}^N z_{ki} \label{LemmaOF}\\
& \text{s.t.} \enspace z_{ki} \left( \eta - \widehat{\xi}_{ki} \right) \geq 0 \enspace \forall i \in \left\lbrace 1, ..., N \right\rbrace, \label{Conspos}\\
& \hphantom{\text{s.t.} \enspace} 0 \leq z_{ki} \leq 1 \hphantom{\text{---------}} \forall i \in \left\lbrace 1, ..., N \right\rbrace, \label{Consbound}
\end{align}
\end{subequations}

\noindent where $\eta \in \mathbb{R}$ is the real number that corresponds to the argument of the function and $z_{ki} \in \mathbb{R} \, \forall k \, \forall i$ is the decision variable. 
\end{lemma}
\begin{proof}
See Appendix \ref{ProofLemmaAppendix}.
\end{proof}
%
Based on this analytical link between $\mathbb{C}$ and $\mathbb{Q}$, it is now possible to reformulate the worst-case expectation problem \eqref{DRWCE} with the ambiguity set $\mathcal{M}_2$. 

\setcounter{theorem}{0}
\color{black}
\begin{theorem}\label{TheoremCopula}

Given $N$ historical observations $\widehat{\xi}_i, \ i \in \left\lbrace 1, ..., N \right\rbrace$ of the random variable $\widetilde{\xi}$, the worst-case expectation problem $\underset{\mathbb{Q} \in \mathcal{M}_2}{\max} \enspace \mathbb{E}^\mathbb{Q} \left[ a \left( x \right)^\top \widetilde{\xi} + b \left( x \right) \right]$ is equivalent to the following conic reformulation:
\begingroup
\allowdisplaybreaks
\begin{subequations}\label{FinalTheorem}
\begin{align} 
& \min_{\Pi} \enspace \alpha \theta_1 + \beta \theta_2 + \frac{1}{N} \sum\limits_{i = 1}^N y_i \label{TheoremOF} \\
& \hphantom{\text{---}} \text{s.t.} \enspace y_i \geq  \max_{\xi \in \Xi} \enspace a \left( x \right)^\top \xi + b \left( x \right) - {\zeta_{i}^{(1)}}^\top \left( \widehat{\xi}_i - \xi \right) \notag \\
& \hphantom{\text{------------------}} - {\zeta_{i}^{(2)}}^\top \left( F \left( \widehat{\xi}_i \right) - F \left( \xi \right)  \right) \enspace \forall i  \label{eq1} \\
& \hphantom{\text{---} \text{s.t.} \enspace} \norm{ \zeta_{i}^{(1)} }_\ast \leq \alpha \enspace \forall i  \label{alphaconstraint} \\
& \hphantom{\text{---} \text{s.t.} \enspace} \norm{ \zeta_{i}^{(2)} }_\ast \leq \beta \enspace \forall i, \label{betaconstraint}
\end{align}
\end{subequations}
\endgroup


\noindent where $i \in \left\lbrace 1, ..., N\right\rbrace$. The decision variables are collected in $\Pi = \left\lbrace x \right.$, $\alpha$, $\beta$, $y_i$, $\zeta_{i}^{(1)}$, $\zeta_{i}^{(2)} \left. \right\rbrace$. In particular, $\alpha$, $\beta \in \mathbb{R}^+$, $y_i \in \mathbb{R} \enspace \forall i \in \left\lbrace 1, ..., N \right\rbrace$, and $\zeta_{i}^{(1)}$, $\zeta_{i}^{(2)} \in \mathbb{R}^{\vert \mathcal{W} \vert}$ are auxiliary variables. Recall that parameters $\theta_1$ and $\theta_2$  in \eqref{TheoremOF} are the Wasserstein radii defined for the ambiguity set $\mathcal{M}_2$. In addition, $F(\widehat{\xi}_i) = \left( F_1(\widehat{\xi}_{1i}), ..., F_k(\widehat{\xi}_{ki}), ..., F_{\vert \mathcal{W} \vert}(\widehat{\xi}_{\vert \mathcal{W} \vert i}) \right)^\top$ and the operator $|| . ||_\ast$ computes the dual norm of a vector. The vector $F \left( \xi \right) = \left(  F_1 \left( \xi_1 \right), ..., F_{\vert \mathcal{W} \vert} \left( \xi_{\vert \mathcal{W} \vert} \right) \right)^\top$ is reformulated using Lemma 1, and is given by

\small
\begin{equation}
    F \left( \xi \right) = \left( \begin{aligned}
       & \left\lbrace \begin{aligned}
            & \underset{z_{jki}}{\max} \enspace \frac{1}{N} \sum\limits_{j=1}^N z_{jki} \\
            & \text{s.t.} \enspace z_{jki} \left( \xi_k - \widehat{\xi}_{kj} \right) \geq 0 \enspace \forall j \\
            & \hphantom{\text{s.t.}} \enspace 0 \leq z_{jki} \leq 1 \enspace \forall j.
        \end{aligned} \right\rbrace , k = 1 \\
        & \hphantom{\text{-----------------------}} \vdots\\
        & \left\lbrace \begin{aligned}
            & \underset{z_{jki}}{\max} \enspace \frac{1}{N} \sum\limits_{j=1}^N z_{jki} \\
            & \text{s.t.} \enspace z_{jki} \left( \xi_k - \widehat{\xi}_{kj} \right) \geq 0 \enspace \forall j \\
            & \hphantom{\text{s.t.}} \enspace 0 \leq z_{jki} \leq 1 \enspace \forall j.
        \end{aligned} \right\rbrace , k = \vert \mathcal{W} \vert \\
    \end{aligned} \right). \label{OptiVector}
\end{equation}  
\end{theorem}
\normalsize

\begin{proof}
See Appendix \ref{ProofTheoremAppendix}.
\end{proof}


\color{black}
The outcome of Theorem 1 is an optimization problem with constraints involving optimization operators, which is not straightforward to be solved by using off-the-shelf solvers. In Appendix \ref{SolutionApproachAppendix}, we aim to reformulate problem \eqref{FinalTheorem}, especially constraint \eqref{eq1}, such that it can be incorporated into an optimization problem. To do so, we leverage duality theory and McCormick relaxation of bilinear terms. This will eventually enable us to reformulate the inner worst-case expectation in \eqref{DRWCE}. \color{black}


\vspace{-0.3cm}
\section{A Power System Application}
\label{Sect:ModelFormulation}

Theorem \ref{TheoremCopula} and its tractable solution approach in Appendix \ref{SolutionApproachAppendix} provide a conic reformulation for a general decision-making problem under uncertainty, and therefore it can be applied to any decision-making optimization problem in power systems. \color{black} As an example, we apply our proposed model to the day-ahead distributionally robust OPF problem, while capturing the dependencies among renewable energy sources. \color{black}We derive the reformulation in two particular cases, i.e., \textit{i)} meshed transmission systems using a DC power flow approximation\footnote{\color{black}We highlight that the DC model of power flows \cite{Christie2000} is a linear and convex approximation to the original complete set of non-linear and non-convex equations governing the AC power flow model. The resulting DC power flow model relies on assumptions that are usually deemed comparatively less strong for high-voltage transmission lines, i.e., lossless power lines, constant voltage magnitudes and small voltage angle differences along a line. However, these assumptions do not guarantee that the solution obtained is necessarily feasible with respect to the AC power flow equations \cite{Baker}.}, and \textit{ii)} radial distribution systems using a \textit{LinDistFlow} approximation of AC power flow equations\footnote{\color{black}The solution of the DC-OPF problem may not be feasible with respect to the complete set of AC power flow equations \cite{Baker}. As a prospect to the contribution in this work, we highlight the extension of our work to the most complete set of AC power flow equations. In partial fulfilment to this prospect, we consider the \textit{LinDistFlow} approximation of AC power flow equations, which is valid for radial distribution systems, and show in the numerical results that the proposed copula-based approach for considering dependencies between wind farms outperforms the traditional techniques which disregard the correlations.}$^{\text{,}}$\footnote{\color{black} To the best of our knowledge, this is the first time in the literature that a metric-based distributionally robust OPF problem is derived using \textit{LinDistFlow} approximation of AC power flow equations (please see \cite{MoB2} for a moment-based counterpart). The reason for this is that the resulting distributionally robust models are usually highly complex to solve. The complexity arises from the additional modeling layer required by DRO, where the operating constraints such as transmission line capacity limits need to be formulated in an appropriate linear or convex form for introducing uncertain fluctuations. \color{black} Any additional contribution towards using more extended AC modeling of power flows, e.g., those investigated in \cite{GOComp}, within the proposed distributionally robust framework would require advanced convexification techniques dedicated to DRO, which is left for future work.\color{black}}.\color{black}



\vspace{-0.2cm}
\subsection{\color{black} The OPF Problem for Meshed Transmission Systems \color{black}}


Given a generic ambiguity set $\mathcal{M}$ (which is either  $\mathcal{M}_1$ or $\mathcal{M}_2$ introduced in Section II and III), \color{black} the distributionally robust day-ahead optimal power flow problem \color{black} for meshed transmission systems \color{black} reads as 
\begingroup
\allowdisplaybreaks
\begin{subequations}\label{DR}
\begin{align}
& \underset{g, \overline{r}, \underline{r}, V}{\text{min}} \enspace c^\top g + \overline{c}^\top \overline{r} + \underline{c}^\top \underline{r} + \underset{\mathbb{Q} \in \mathcal{M}}{\text{max}} \enspace \mathbb{E}^{\mathbb{Q}} \left[ c^\top V \widetilde{\xi}  \right]  \label{DROF} \\
& \text{s.t.} \enspace g + \overline{r} \leq g^{\text{max}}, \enspace g - \underline{r} \geq g^{\text{min}}, \label{DRPmin}\\
& \hphantom{\text{s.t.} \enspace} 0 \leq \underline{r} \leq r^{\text{max}}, \enspace 0 \leq \overline{r} \leq r^{\text{max}}, \label{DRRmax}\\
& \hphantom{\text{s.t.} \enspace} \mathds{1}^\top g + \mathds{1}^\top W \mu - \mathds{1}^\top d = 0, \label{DRDAbalance}\\
& \hphantom{\text{s.t.} \enspace} \sum\limits_{p \in \mathcal{P}} V_{p,w} + W_{w,w} = 0 \enspace \forall w \in \mathcal{W}, \label{DRRTbalance} \\
& \hphantom{\text{s.t.} \enspace} \underset{\mathbb{Q} \in \mathcal{M}}{\text{min}} \enspace  \mathbb{Q} \left( -\underline{r}_p \leq V_p \widetilde{\xi} \right) \geq 1 - \underline{\epsilon}_p \enspace \forall p \in \mathcal{P}, \label{DRRTdownflex}\\
& \hphantom{\text{s.t.} \enspace} \underset{\mathbb{Q} \in \mathcal{M}}{\text{min}} \enspace \mathbb{Q} \left( V_p \widetilde{\xi} \leq \overline{r}_p \right) \geq 1 - \overline{\epsilon}_p \enspace \hphantom{\mathbf{-}} \forall p \in  \mathcal{P}, \label{DRRTupflex} \\
& \hphantom{\text{s.t.} \enspace} \underset{\mathbb{Q} \in \mathcal{M}}{\text{min}} \enspace \mathbb{Q} \left( \text{T}_f^{\mathcal{P}} \left( g + V \widetilde{\xi} \right) + \text{T}_f^{\mathcal{W}} W \left( \mu + \widetilde{\xi} \right) \right. \notag \\ 
& \hphantom{\text{s.t.} \enspace  ---} \left.  \left. - \text{T}_f^{\mathcal{D}} d \right) \leq f^{\text{max}}_f \right) \geq 1 - \epsilon_f \enspace \enspace \hphantom{\mathbf{-}} \forall f \in  \mathcal{F}. \label{DRRTFmax}
\end{align}
\end{subequations}
\endgroup

\color{black}

The problem \eqref{DR} optimizes the day-ahead dispatch $g \in \mathbb{R}^{\vert \mathcal{P} \vert}$ of conventional generating units as well as their upward and downward reserve capacity $\overline{r} \in \mathbb{R}^{\vert \mathcal{P} \vert}$ and $\underline{r} \in \mathbb{R}^{\vert \mathcal{P} \vert}$ to be booked in the day-ahead stage.
Objective function \eqref{DROF} minimizes the total operational cost of the system. The first three terms in \eqref{DROF} are linear and contains the energy production cost $c \in \mathbb{R}^{\vert \mathcal{P} \vert}$ and the upward/downward reserve procurement cost $\overline{c} \in \mathbb{R}^{\vert \mathcal{P} \vert}$ and $\underline{c} \in \mathbb{R}^{\vert \mathcal{P} \vert}$. The fourth term in \eqref{DROF}, i.e., $\underset{\mathbb{Q} \in \mathcal{M}}{\text{max}} \enspace \mathbb{E}^{\mathbb{Q}} \left[ c^\top V \widetilde{\xi}  \right]$, refers to the operational cost of conventional units, incurred by the activation of their reserve capacity in the real-time operation. This cost is calculated in expectation with respect to the worst-case distribution $\mathbb{Q}$, which is endogenously selected within the ambiguity set $\mathcal{M}$. It is approximated by linear decision rules \cite{LDR} using matrix $V \in \mathbb{R}^{\vert \mathcal{P} \vert \times \vert \mathcal{W} \vert}$, whose elements are decision variables. \color{black} The uncertainty in the distributionally robust OPF problem arises from renewable wind power in-feed, when the real-time realization deviates from its day-ahead forecast. We model the renewable wind power in-feeds as
\begin{equation}
\omega = \mu + \tilde{\xi}, \label{wmuxi}
\end{equation}
\noindent where $ \omega \in \mathbb{R}^{\vert \mathcal{K} \vert}$ describes the actual real-time wind power generation, $ \mu \in \mathbb{R}^{\vert \mathcal{K} \vert} $ defines the day-ahead forecast for wind power generation and $ \tilde{\xi} \in \mathbb{R}^{\vert \mathcal{K} \vert}$ is the deviation in real-time from day-ahead forecast\footnote{\color{black}We assume a non-biased forecaster, such that the average forecast error $\mathbb{E} \left[ \tilde{\xi} \right]$ equals 0. Furthermore, we assume that the wind power units are dispatched in day-ahead at a value equal to the single-point forecast $\mu$.\color{black}}$^{,}$\footnote{\color{black}An alternative modeling approach would consider the renewable wind power units as dispatchable units which requires an additional decision variable modeling their day-ahead schedule. The main difference of this alternative compared to \eqref{wmuxi} pertains to the source of uncertainty (i.e., forecast errors in one case, and actual real-time wind power generation in the other one). We hypothesize that the innovation of incorporating the dependence information about the uncertain parameters into the definition of ambiguity set remains beneficial in both cases. The copula-based restriction results in an enhanced representation of dependencies whatsoever the underlying source of uncertainty is. Therefore, the choice of modeling approach will have a limited impact on the final conclusions of the paper.\color{black}}. In the following, we consider $\tilde{\xi}$ as the uncertain parameter in the distributionally robust OPF problem, and model it as a random variable described by the probability distribution $\mathbb{Q}$. \color{black} Note that the renewable production cost is assumed to be zero.\color{black}

Constraints \eqref{DRPmin} ensure that the generation level of conventional units lies within their maximum and minimum limits, i.e., $g^{\text{max}} \in \mathbb{R}^{\vert \mathcal{P} \vert}$ and $g^{\text{min}} \in \mathbb{R}^{\vert \mathcal{P} \vert}$. Constraints \eqref{DRRmax} enforce the maximum amount of reserve capacity $r^{\text{max}} \in \mathbb{R}^{\vert \mathcal{P} \vert}$ that can be provided by conventional units. Constraints \eqref{DRDAbalance} and \eqref{DRRTbalance}  ensure the day-ahead and real-time power balance, respectively. In particular, the day-ahead constraint  \eqref{DRDAbalance} enforces the sum of total production of conventional generating units $\mathds{1}^\top g$ and total forecasted production of renewable units $\mathds{1}^\top W \mu$ to be equal to total demand $\mathds{1}^\top d$. Note that the demands $d \in \mathbb{R}^{\vert \mathcal{D} \vert}$ are assumed inelastic to prices. In addition, $W \in \mathbb{R}^{\vert \mathcal{W} \vert \times \vert \mathcal{W} \vert}$ is a diagonal matrix of the installed capacity of renewable power units, whereas $\mu \in \mathbb{R}^{\vert \mathcal{W} \vert}$ gives their per-unit power generation forecast. The real-time balance in \eqref{DRRTbalance} is ensured via the elements of matrix $V$, the so-called participation factors, which can be interpreted as follows. The conventional unit $p$ responds to any deviation in renewable power generation of farm $w$, i.e., $\widetilde{\xi}_w$ (with respect to the day-ahead forecast $\mu_w$) based on its corresponding participation factors $V_{p,w}$. Therefore, for renewable power deviation $\widetilde{\xi}_w$, the recourse action of conventional unit $p$ is  $V_{p,w} \widetilde{\xi}_w$, such that the total recourse action $\sum_{p \in \mathcal{P}} V_{p,w} \widetilde{\xi}_w$ compensates the total deviation $W_{w,w} \widetilde{\xi}_w$. Note that \eqref{DRRTbalance} is an equality constraint, so $\widetilde{\xi}_w$ can be dropped from both sides. In the real-time operation, the recourse action $V_p \widetilde{\xi} = \sum_{w \in \mathcal{W}} V_{p,w} \widetilde{\xi}_w$ of conventional units should be limited to the reserve capacities procured in the day-ahead stage, i.e., $\overline{r}_p$ and $\underline{r}_p$. Similarly, the real-time flow within each transmission line $f \in \mathcal{F}$ should respect the capacity $f_f^\text{max}$. The real-time flows are expressed using the power transfer distribution factor matrices $\text{T}_f^{\mathcal{P}} \in \mathbb{R}^{\vert \mathcal{F} \vert \times \vert \mathcal{P} \vert}$, $\text{T}_f^{\mathcal{W}} \in \mathbb{R}^{\vert \mathcal{F} \vert \times \vert \mathcal{W} \vert}$ and $\text{T}_f^{\mathcal{D}} \in \mathbb{R}^{\vert \mathcal{F} \vert \times \vert \mathcal{D} \vert}$ for conventional generating units, renewable energy sources, and demands, respectively. Note that these restrictions are enforced via via probabilistic constraints, namely, Distributionally Robust Chance Constraints (DRCCs), for which we provide further details later in Section \ref{SubSect:Refor}. 

\color{black}
\subsection{The OPF Problem for Radial Distribution Systems}

In this section, we derive the distributionally robust day-ahead OPF problem for radial distribution systems using the \textit{LinDistFlow} approximation of AC power flow equations \cite{Lindistflow}. Given the generic ambiguity set $\mathcal{M}$, the model reads as

\begingroup
\allowdisplaybreaks
\begin{subequations}\label{DRACOPF}
\begin{align}
& \min _{g^{\dagger}, f^{\dagger}, V, u} \enspace \sum_{i \in \mathcal{N}} c_{i} g_{i}^{P}  + \underset{\mathbb{Q} \in \mathcal{M}}{\text{max}} \enspace \mathbb{E}^{\mathbb{Q}} \left[ \sum_{i \in \mathcal{N}} c_{i} V_{i} \widetilde{\xi}^P  \right] \label{ACDROPF_OF} \\
\text { s.t. } & g_{0}^{\dagger}=\sum_{i \in \mathcal{D}_{0}}\left(d_{i}^{\dagger}-g_{i}^{\dagger}\right)-\mu_i^\dagger, \, \dagger \in \left\lbrace P, Q \right\rbrace, \label{ACDROPF_1}\\
& u_{0}=1, \label{ACDROPF_2} \\
& f_{l}^{\dagger}=\sum_{i \in \mathcal{D}_{l}}\left(d_{i}^{\dagger}-g_{i}^{\dagger}-\mu_i^\dagger \right), \, \dagger \in \left\lbrace P, Q \right\rbrace, \forall l \in \mathcal{L}, \label{ACDROPF_3}\\
& u_{i}=u_{0}-2 \sum_{l \in \mathcal{R}_{i}}\left(f_{l}^{P} R_{l}+f_{l}^{Q} X_{l}\right), \forall i \in \mathcal{L}, \label{ACDROPF_4} \\
&\left(f_{l}^{P}\right)^{2}+\left(f_{l}^{Q}\right)^{2} \leq \bar{f}_{l}^{2}, \enspace \forall l \in \mathcal{L}, \label{ACDROPF_5} \\
& \sum\limits_{i \in \mathcal{N}} V_{i,w} + 1 = 0 \enspace \forall w \in \mathcal{W}, \label{DRRTPbalance} \\
& \underset{\mathbb{Q} \in \mathcal{M}}{\text{min}} \enspace \mathbb{Q} \left( \widetilde{g_{i}}^{\dagger}  \leq \bar{g}_{i}^{\dagger} \right) \geq 1 - \underline{\epsilon}_i^{\dagger}, \, \dagger \in \left\lbrace P, Q \right\rbrace, \enspace \forall i \in \mathcal{N}, \label{ACDROPF_6}\\
& \underset{\mathbb{Q} \in \mathcal{M}}{\text{min}} \enspace \mathbb{Q} \left( \underline{g}_{i}^{\dagger} \leq \widetilde{g_{i}}^{\dagger}  \right) \geq 1 - \underline{\epsilon}_i^{\dagger}, \, \dagger \in \left\lbrace P, Q \right\rbrace, \enspace \forall i \in \mathcal{N}, \label{ACDROPF_7}\\
& \underset{\mathbb{Q} \in \mathcal{M}}{\text{min}} \enspace \mathbb{Q} \left( \widetilde{u_{i}} \leq \bar{v}_{i}^{2} \right) \geq 1 - \underline{\epsilon}_i, \enspace \forall i \in \mathcal{N}_{\backslash 0},\label{ACDROPF_8} \\
& \underset{\mathbb{Q} \in \mathcal{M}}{\text{min}} \enspace \mathbb{Q} \left( \underline{v}_{i}^{2} \leq \widetilde{u_{i}} \right) \geq 1 - \underline{\epsilon}_i, \enspace \forall i \in \mathcal{N}_{\backslash 0},\label{ACDROPF_9}
\end{align}
\end{subequations}
\endgroup

\noindent where $\dagger = \left\lbrace P, Q\right\rbrace$ represents the superscripts related to, respectively, active and reactive power. The set $\mathcal{N}$ collects the nodes of the radial distribution system and the sets $\mathcal{D}_i$ and $\mathcal{R}_i$ respectively include the downstream and root nodes to node $i$. The set $l \in \mathcal{L}$ represents the set of distribution lines. Given that, the problem \eqref{DRACOPF} seeks the optimal generation dispatch of controllable generators $g^P \in \mathbb{R}^{\vert \mathcal{N} \vert}$ at the distribution level. The objective function \eqref{ACDROPF_OF} minimizes the total operating cost, composed of the day-ahead energy production cost $c \in \mathbb{R}^{\vert \mathcal{N} \vert}$ and the real-time operating cost of recourse actions. Similarly to model (10), the real-time cost is calculated in expectation for the worst-case distribution $\mathbb{Q}$ of wind power generation forecast errors $\widetilde{\xi}^\dagger \in \mathbb{R}^{\vert \mathcal{W} \vert}$ in the ambiguity set $\mathcal{M}$, and the recourse actions are modeled via linear decision rules, using $V \in \mathbb{R}^{\vert \mathcal{N} \vert \times \vert \mathcal{W} \vert}$. 

Constraint \eqref{ACDROPF_1} calculates the required active and reactive power injections $g_0^P \in \mathbb{R}$ and $g_0^Q \in \mathbb{R}$ at the interface between transmission and distribution systems (node 0) to balance the mismatch between consumption $d^\dagger \in \mathbb{R}^{\vert \mathcal{N} \vert}$ and generation $g^\dagger \in \mathbb{R}^{\vert \mathcal{N} \vert}$ in the downstream distribution system. Constraint \eqref{ACDROPF_2} imposes the voltage magnitude $u_0$ at the slack node, to be equal to 1 per-unit. Constraint \eqref{ACDROPF_3} calculates the active and reactive power flows $f_l^P \in \mathbb{R}$ and $f_l^Q \in \mathbb{R}$ in the distribution lines $l \in \mathcal{L}$. Constraint \eqref{ACDROPF_4} defines the per-unit square voltage $u_i \in \mathbb{R}$ at each node of the distribution system, given the per-unit resistances $R_l$ and reactances $X_l$, and the active and reactive power flows $f_l^P$ and $f_l^Q$ in the root lines. The apparent power flow $f_{l}^{2} \in \mathbb{R}$ is limited to $\bar{f}_{l}^{2} \in \mathbb{R}$ via \eqref{ACDROPF_5} for all lines in the distribution system. Constraints \eqref{DRRTPbalance} to \eqref{ACDROPF_8} impose the real-time operational restrictions. Constraint \eqref{DRRTPbalance} enforces the real-time balancing between generation and consumption, using the linear decision rules, similarly to model (10). The minimum and maximum operating limits of active and reactive power generation, and voltage magnitude are enforced via the DRCCs \eqref{ACDROPF_6} to \eqref{ACDROPF_9}. The rationale behind this is that the voltage magnitude and the active and reactive power in-feeds are real-time state variables that depend on the wind power uncertainty $\widetilde{\xi}^\dagger$. To appropriately model these real-time variations, we further define the real-time power generation of controllable generators as 
\begingroup
\begin{subequations}
\begin{align}\label{Addon1}
    \widetilde{g}_i^\dagger = g_i^\dagger + V_i \widetilde{\xi}^\dagger, \enspace \dagger \in \left\lbrace P, Q \right\rbrace,
\end{align}
\end{subequations}
\endgroup
\noindent where the first term refers to the day-ahead dispatch $g_i^\dagger$ and the second term refers to the real-time response $V_i \widetilde{\xi}^\dagger$ to deviations $\widetilde{\xi}^\dagger$, using the matrix of participation factors $V \in \mathbb{R}^{\vert \mathcal{G} \vert \times \vert \mathcal{W} \vert}$. In addition, we model the real-time flows $\widetilde{f}_l^P$ and $\widetilde{f}_l^Q$ as follows: 
\begingroup
\begin{subequations}
\begin{align}\label{Addon2}
    & \widetilde{f}_l^\dagger = f_l^\dagger + \sum_{i \in \mathcal{D}_{l}} \left( -V_i \widetilde{\xi}^\dagger - \widetilde{\xi}^\dagger_l \right), \enspace \dagger \in \left\lbrace P, Q \right\rbrace,
\end{align}
\end{subequations}
\endgroup
\noindent where the power flow $ f_l^\dagger $ scheduled in day-ahead is supplemented with the corresponding real-time deviation, composed of the deviation in renewable power generation and the compensation made by the controllable generators in the downstream nodes. Finally, the real-time square voltage $\widetilde{u}_i$ is expressed as follows: 
\begingroup
\begin{subequations}
\begin{align}\label{Addon3}
    \widetilde{u}_i = u_0 -2 \sum_{l \in \mathcal{R}_{i}}\left(\widetilde{f}_{l}^{P} R_{l}+\widetilde{f}_{l}^{Q} X_{l}\right),
\end{align}
\end{subequations}
\endgroup
\noindent where the real-time expression of power flows are used. The combination of \eqref{Addon1}, \eqref{Addon2} and \eqref{Addon3} with the model \eqref{DRACOPF} results in the complete distributionally robust OPF formulation for radial distribution systems. \color{black}

\subsection{Reformulation of \eqref{DR} and \eqref{DRACOPF}}
\label{SubSect:Refor}
\color{black}

The inequalities \eqref{DRRTdownflex} to \eqref{DRRTFmax} as well as \eqref{ACDROPF_6} to \eqref{ACDROPF_9} are enforced via DRCCs. These DRCCs state that the probabilistic constraints within parentheses should be respected under the worst-case distribution $\mathbb{Q}$ within the ambiguity set $\mathcal{M}$ with a probability not lower than $1-\epsilon$. Note that the value of parameter $\epsilon \in \mathbb{R}$ lies between zero and one, fixed by the power system operator. 

The procedure to reformulate these DRCCs is the same for all inequalities considered in both power system applications. For this purpose, we use a Conditional-Value-at-Risk (CVaR) approximation \cite{Zymler}. \color{black} Consider a generic DRCC in the form of 
\begin{subequations}
\begin{align}
\underset{\mathbb{Q} \in \mathcal{M}}{\text{min}} \enspace \mathbb{Q} \left( . \leq 0 \right) \geq 1 - \epsilon. \label{CVaRReforq}
\end{align}

In order to get rid of the probability operator $\mathbb{Q} \left( . \right)$, the DRCC \eqref{CVaRReforq} can be approximated by the following CVaR constraint:
\begin{align}
\underset{\mathbb{Q} \in \mathcal{M}}{\text{max}} \; \mathbb{Q}\text{-CVaR}_\epsilon ( . ) \leq 0. \label{CVaRReforqw}
\end{align}

The CVaR operator in the left-hand side of \eqref{CVaRReforqw} is defined as 
\begin{align}
\underset{\tau \in \mathbb{R}}{\text{min}} \enspace \tau + \frac{1}{\epsilon} \enspace \underset{\mathbb{Q} \in \mathcal{M}}{\text{max}} \enspace \mathbb{E}^\mathbb{Q} \left[ \lceil . - \tau \rceil^+ \right],  \label{CVaRRefor}
\end{align}
\end{subequations}
\noindent where $\tau \in \mathbb{R}$ is an auxiliary variable and $\lceil . \rceil^+ = \max \left( ., 0\right)$. By this approximation, the worst-case expectation problem  appears not only in the objective function \eqref{DROF}, but also in every approximated DRCC in the form of \eqref{CVaRRefor}. Note that the inequalities \eqref{DRRTdownflex} to \eqref{DRRTFmax} as well as \eqref{ACDROPF_6} to \eqref{ACDROPF_9} are linear. Therefore, the worst-case expectation problem appearing in \eqref{CVaRRefor} complies with the setting of Theorem 1. \color{black} In the following, we provide all reformulations required to solve the distributionally robust OPF problem under both ambiguity sets $\mathcal{M}_1$ and $\mathcal{M}_2$, in a generic way such that it can be applied to both DC and AC formulations. \color{black}




\subsubsection{Metric-Based Ambiguity Set $\mathcal{M}_1$}
\label{SubSect:MetricBasedAmbiguitySet}

Reference \cite{Kuhn} provides the reformulation of a generic worst-case expectation problem in the form of $\underset{\mathbb{Q} \in \mathcal{M}_1}{\max} \enspace \mathbb{E}^\mathbb{Q} \left[ a \left( x \right)^\top \widetilde{\xi} + b \left( x \right) \right]$. Similarly to  Theorem \ref{TheoremCopula}, this maximization over probability distributions $\mathbb{Q} \in \mathcal{M}_1$ can be recast into a minimization problem at the cost of introducing a set of additional auxiliary variables. With this reformulation, the objective function \eqref{DROF} \color{black} or \eqref{ACDROPF_OF} \color{black} would contain two min operators that can be merged, yielding a formulation that can be fed into an off-the-shelf solver. Next, we reformulate the worst-case expectation appeared in constraints \eqref{DRRTupflex} to \eqref{DRRTFmax} after using the CVaR approximation. Following \cite{Kuhn}, the minimization operators appeared in the constraints can be eventually dropped, but again at the cost of additional auxiliary variables. For the sake of conciseness and completeness, we provide the reformulation of a generic worst-case expectation problem, as well as the final reformulation of \color{black} problems \eqref{DR} and \eqref{DRACOPF} \color{black} under ambiguity set $\mathcal{M}_1$ in the online companion  \cite{OnlineAppendix}. 

\subsubsection{Copula-Based Ambiguity Set $\mathcal{M}_2$}
\label{SubSect:CopulaBasedAmbiguitySet}

The general procedure to reformulate \color{black} \eqref{DR} and \eqref{DRACOPF} \color{black} under ambiguity set $\mathcal{M}_2$ is similar to the one under set $\mathcal{M}_1$. However, instead of following the reformulations in \cite{Kuhn}, we apply the outcomes of Theorem \ref{TheoremCopula} in Section \ref{Sect:Reformulation}. Recall that the solution approach for Theorem \ref{TheoremCopula} proposed in Appendix \ref{SolutionApproachAppendix}, allows us to recast the maximization problem over probability distributions $\mathbb{Q} \in \mathcal{M}_2$ as a minimization problem, but at the cost of additional auxiliary variables. We provide the final reformulation of \color{black} problems \eqref{DR} and \eqref{DRACOPF} \color{black} under ambiguity set $\mathcal{M}_2$  in the online companion  \cite{OnlineAppendix}.

\color{black}
\section{Numerical Experiments}
\label{Sect:NumericalStudy}

We first perform numerical experiments for a meshed transmission system considering the distributionally robust OPF problem \eqref{DR}. We discuss the procedure of out-of-sample analysis, the computational performance and the out-of-sample performance. Next, we introduce a distribution system case study and derive similar experiments considering the distributionally robut OPF problem \eqref{DRACOPF}. 

Our meshed transmission system case study is based on 
\color{black} a slightly updated version of the $24$-node IEEE reliability test system \cite{24IEEEChristos}, whose input data are provided in the online companion \cite{OnlineAppendix}. This case study is composed of $12$ conventional generating units with an aggregate capacity of $2,362.5$ MW, two wind farms with a total maximum installed capacity of $1,000$ MW, and eventually $17$ loads  with an aggregate demand of $2,207$ MW. \color{black} These power suppliers and demands are connected through a network composed of 24 nodes and 34 transmission lines. \color{black}The sole source of uncertainty is the deviation of renewable power generation in the real-time operation with respect to the day-ahead forecast values. To cope with such an uncertainty, the system operator reserves a fraction of capacities of conventional units in the day-ahead stage, and activates them in the real-time operation, if necessary. The total maximum reserve capacity that  conventional units can provide  is $798$ MW. 



\color{black}
To assess the impacts of dependence structure on operational decisions of the system operator, \color{black} we generate a dataset of $1,000$ samples representing the historical wind power observations. To do so, we use the package \textit{DatagenCopulaBased} v1.3.0 in Julia programming language v1.4.2, which allows to generate samples with a predefined dependence structure (e.g., following the Gaussian copula). The resulting dataset mimics historical wind power observations from which we retrieve the mean value $\mu$ corresponding to the day-ahead forecast. The final dataset of forecast errors and its corresponding copula will be illustrated later in Fig. \ref{Independence}.
\color{black}

\subsection{Procedure of the Out-of-Sample Analysis}
\label{SubSect:OSA}
Aiming to conduct an ex-post out-of-sample analysis and therefore to compare different models on a fair basis, we split the dataset with $1,000$ samples into two different sets of samples. The first one contains $30$ samples only (the in-sample data), which are used to characterize the uncertain wind power generation within the models. Therefore, $N=30$ and indices $i$ and $j$ in \eqref{final2} run from $1$ to $30$. The remaining $970$ samples in the dataset are used as unseen wind power realizations to assess the quality of  operational decisions made in the day-ahead stage. 

\color{black} The out-of-sample analysis is as follows. We first solve problem \eqref{DR} under ambiguity sets $\mathcal{M}_1$ and $\mathcal{M}_2$, fed with in-sample dataset\footnote{\color{black} Recall that the ambiguity set $\mathcal{M}_1$ defines the state-of-the-art Wasserstein metric-based ambiguity set, which serves as the benchmark solution in the subsequent numerical experiments, for assessing the performances of the proposed copula-based ambiguity set $\mathcal{M}_2$ under various values of $\theta_1$ and $\theta_2$.}, and obtain the optimal day-ahead decisions $g$, $\underline{r}$, and $\overline{r}$. \color{black} Given fixed values of those day-ahead decisions, we then solve a deterministic optimization problem in the real-time operation for each of $970$ unseen samples, whose outcomes are the recourse action of conventional generating units as well as the involuntarily load shedding and the wind curtailment as two extreme recourse actions. 
The involuntarily load shedding is required when there is a wind power deficit in the real time, and that the upward reserve capacities provided by conventional units are insufficient to compensate the entire deficit. Similarly, the wind curtailment occurs when there is a wind power excess in the real time, combined with the lack of downward reserve capacities provided by conventional units to absorb all available wind power. The formulation of this deterministic model is available in the online companion \cite{OnlineAppendix}. Once this deterministic optimization problem is solved $970$ times for each model, we calculate the average real-time operational cost of the system. In the rest of this section, we report the out-of-sample cost, which is the sum of operational cost of the system in the day-ahead stage, i.e., the first three terms in \eqref{DROF}, and the average real-time operational cost calculated from the out-of-sample analysis.



\subsection{Computational Performance}
\label{SubSect:ComputationalIssue}

All models are solved using Gurobi v8.0.1 in JuMP v0.21.3 under programming language JuliaPro v1.4.2 on a usual computer clocking at $2.2$ GHz with $16$ GB of RAM. All source codes are publicly available in the online companion \cite{OnlineAppendix}.

\color{black} Fig. \ref{TimeVSN} shows the computational time as a function of the number of historical in-sample observations $N$ under different settings. We observe that the computational time increases with the number of historical observations in a non-linear manner. These results suggest that the computational time growth follows a quadratic trend, which is in line with the increase in the number of variables and constraints --- see Theorem 1, where the number of certain variables and constraints increases quadratically\footnote{Compared to other uncertainty modeling techniques, e.g., scenario-based stochastic programming, the DRO approach generally provides more qualified decisions in terms of the out-of-sample performance, when the number of historical in-sample observations is relatively low. This may further motivate the use of DRO when there is limited historical data, or when  the decision-maker aims to reduce the computational time by intentionally reducing the number of historical observations. As shown in Fig. \ref{TimeVSN}, the computational time for the DRO approach is satisfactory when $N$ is comparatively low.} in $N$. In addition, Fig. \ref{CopulaTimes} depicts the computational time for various values of parameters $\theta_1$ and $\theta_2$ in the proposed ambiguity set $\mathcal{M}_2$, and for $\epsilon = 0.01$ and $0.03$. The results highlight that the computational time decreases when $\theta_2$ increases. \color{black} Recall that two wind farms only have been considered so far. \color{black}We will investigate later in Section \ref{Sect:Scalability} the scalability of the proposed model with respect to the dimensionality of the uncertainty space by increasing the number of wind farms\footnote{\color{black}In this paper, we consider the scalability of our proposed framework with respect to the number of in-sample and the number of wind farms, as these indicators directly impact the description of uncertainty. The scales of networks used in these simulations are usually deemed acceptable for numerical experiment purposes, and allow us to bring valuable and informative discussions on the outcomes of the proposed method.}.\color{black} 

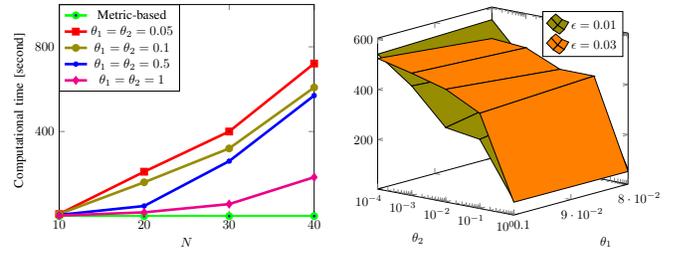
\begin{figure}[]
\subfigure[\color{black} Computational time as a function of the number of historical in-sample observations $N$. Fixed value: $\epsilon = 0.05$.]{\resizebox {0.48\columnwidth} {!} {\begin{tikzpicture} 
\begin{axis}[enlargelimits=false,xmin=10, xmax=40,
    ymin=0, ymax=1000,
    xtick={10,20,30,40,50},
    ytick={400,800,1200,1600,2000},
    xlabel = $N$,
    ylabel = {Computational time [second]},
    legend style={at={(0,1)},anchor=north west},] 
\addplot+[samples=400,color=green,line width = 2pt,] 
coordinates {(10,0)(20,0.1)(30,0.2)(40,0.3)}; 
\addlegendentry{Metric-based}
\addplot+[samples=400,color=red,line width = 2pt,] 
coordinates {(10,10)(20,210)(30,400)(40,721)}; 
\addlegendentry{$\theta_1 = \theta_2 = 0.05$}
\addplot+[samples=400,color=olive,line width = 2pt,] 
coordinates {(10,10)(20,160)(30,320)(40,608)}; 
\addlegendentry{$\theta_1 = \theta_2 = 0.1$}
\addplot+[samples=400,color=blue,line width = 2pt,] 
coordinates {(10,6)(20,48)(30,260)(40,570)}; 
\addlegendentry{$\theta_1 = \theta_2 = 0.5$}
\addplot+[samples=400,color=magenta,line width = 2pt,] 
coordinates {(10,3)(20,18)(30,57)(40,184)}; 
\addlegendentry{$\theta_1 = \theta_2 = 1$}
\end{axis} 
\end{tikzpicture} \label{TimeVSN}}} 
 	\enspace
 	\subfigure[\color{black} Computational time as a function of parameters $\theta_1$ and $\theta_2$. Fixed values: $N = 25$. The metric-based ambiguity set approach finds a solution in less than 1 second for the values of $\theta_1$ explored in this study. ]{\resizebox {0.48\columnwidth} {!} {\begin{tikzpicture}
\begin{semilogyaxis}[view={130}{15}, xlabel=$\theta_1$, ylabel=$\theta_2$]
\addplot3[surf, mesh/cols=5, mesh/rows=2, color=olive, shader=flat, draw = black,] 
file {DataFigure/Eps01.txt};
\addlegendentry{$\epsilon = 0.01$}
\addplot3[surf, mesh/cols=5, mesh/rows=2, color=orange, shader=flat, draw = black,
] 
file {DataFigure/Eps03.txt};
\addlegendentry{$\epsilon = 0.03$}
\end{semilogyaxis}
\end{tikzpicture} \label{CopulaTimes}}}
 	\caption{\color{black} Computational study.} 
 	    \vspace{-0.1cm}
\end{figure}

In the next two subsections, we investigate the out-of-sample performance of the DRO model with the ambiguity sets $\mathcal{M}_1$ and $\mathcal{M}_2$. Recall that the ambiguity set $\mathcal{M}_1$ contains one risk-tuning parameter only\footnote{The violation probability $\epsilon_{\text{(.)}}$ of each DRCC  can also be seen as a risk-tuning parameter. However, we keep this value unchanged (fixed to $0.05$) in our numerical study, and focus on the risk-tuning parameters related to the proposed ambiguity set.}, i.e., $\theta_1$, that restricts the distance of distributions within the ambiguity set to the empirical one. In contrast, the ambiguity set  $\mathcal{M}_2$ comprises of two risk-tuning parameters, i.e., $\theta_1$ and $\theta_2$, where $\theta_2$ restricts the similarity of  distributions within the ambiguity set in terms of the dependence structure to that of the empirical one. 
For the ease of comparison, we  alternately fix one of the two risk-tuning parameters $\theta_1$ and $\theta_2$, while varying the other one.

\begin{figure*}[t!]
\subfigure[Historical wind power generation of two farms (red plot) and the corresponding empirical copula (blue plot)]{\resizebox {0.64\columnwidth} {!} {\begin{tikzpicture} 
\begin{axis}[enlargelimits=false, xmin=0, ymin=0, xmax=1, ymax=1, xlabel={\fontsize{12}{14.4}\selectfont $\mu_1 + \widetilde{\xi}_1$ [p.u.]}, ylabel={\fontsize{12}{14.4}\selectfont $\mu_2 + \widetilde{\xi}_2$ [p.u.]}] 
\addplot+[only marks,samples=400,color=red, mark = +, fill = red] 
file {DataFigure/IndependentCopulaSample.dat};
 \coordinate (insetPosition) at (rel axis cs:1.15,1.15);
\end{axis} 

\begin{axis}[axis background/.style={fill=white!100}, enlargelimits=false,xmin=0, ymin=0, xmax=1, ymax=1,xtick={0,1}, ytick={0,1}, at={(insetPosition)},anchor={outer north east},footnotesize, xlabel=$U_1$, ylabel=$U_2$]
\addplot+[only marks,samples=400,color=blue, mark = +, fill = blue] 
file {DataFigure/F_xi_Independent.dat};
\end{axis}
 
\end{tikzpicture} \label{Independence}}} 
 	\enspace
 	\subfigure[Out-of-sample cost as a function of $\theta_1$ considering different values for $\theta_2$]{\resizebox {0.66\columnwidth} {!} {\begin{tikzpicture} 
\begin{semilogxaxis}[enlargelimits=false,xmin=0.0001, xmax=1,
    ymin=25000, ymax=31000,
    xtick={10e-4,10e-3,10e-2,10e-1,10e0,10e1,10e2,10e3,10e4},
    ytick={25000,26000,27000,28000,29000,30000,31000,32000},
    legend style={at={(0,1)},anchor=north west},
    xlabel=$\theta_1$,
    ylabel= Operational cost of the system (\$)], 
\addplot+[samples=400, color=black] 
file {DataFigureV3/IndependentDROWithoutDependenceN30_Exp2.dat}; 
\addlegendentry{$\mathcal{M}_1$}
\addplot+[samples=400, color=blue] 
file {DataFigureV3/CopulaDependentOPF_Independent_rho2_01_Exp2.dat}; 
\addlegendentry{$\theta_2$ = 0.1}
\addplot+[samples=400, color=brown] 
file {DataFigureV3/CopulaDependentOPF_Independent_rho2_002_Exp2.dat};  
\addlegendentry{$\theta_2$ = 0.02}
\addplot+[samples=400, color=red] 
file {DataFigureV3/CopulaDependentOPF_Independent_rho2_001_Exp2.dat}; 
\addlegendentry{$\theta_2$ = 0.01}

\addplot[name path = A,	opacity = 0] 
	file {DataFigureV3/IndependentDROWithoutDependenceN30_stdp2.dat};
\addplot[name path = D, opacity = 0] 
	file {DataFigureV3/IndependentDROWithoutDependenceN30_stdm2.dat};
\addplot[name path = B,	opacity = 0] 
	file {DataFigureV3/CopulaDependentOPF_Independent_rho2_01_stdm2.dat};
\addplot[name path = C, opacity = 0] 
	file {DataFigureV3/CopulaDependentOPF_Independent_rho2_01_stdp2.dat};
\addplot[name path = E,	opacity = 0] 
	file {DataFigureV3/CopulaDependentOPF_Independent_rho2_001_stdm2.dat};
\addplot[name path = F, opacity = 0] 
	file {DataFigureV3/CopulaDependentOPF_Independent_rho2_001_stdp2.dat};
\addplot[name path = H,	opacity = 0] 
	file {DataFigureV3/CopulaDependentOPF_Independent_rho2_002_stdm2.dat};
\addplot[name path = I, opacity = 0] 
	file {DataFigureV3/CopulaDependentOPF_Independent_rho2_002_stdp2.dat};

\addplot [blue,fill=black!90!black,opacity=0.3]
    fill between[of=A and D];
\addplot [blue,fill=blue!90!black,opacity=0.3]
    fill between[of=B and C];
\addplot [red,fill=red!90!black,opacity=0.3]
    fill between[of=E and F];
\addplot [brown,fill=brown!90!black,opacity=0.3]
    fill between[of=H and I];

\end{semilogxaxis} 
\end{tikzpicture} \label{Indtheta2}}}
\enspace
 	\subfigure[Out-of-sample cost as a function of $\theta_2$ considering different values for $\theta_1$]{\resizebox {0.66\columnwidth} {!} {\begin{tikzpicture} 
\begin{semilogxaxis}[enlargelimits=false,xmin=0.0001, xmax=1,
    ymin=25000, ymax=30000,
    xtick={10e-4,10e-3,10e-2,10e-1,10e0,10e1,10e2,10e3,10e4},
    ytick={25000,26000,27000,28000,29000,30000,31000,32000},
    legend style={at={(0,0)},anchor=south west},
    xlabel=$\theta_2$,
    ylabel= Operational cost of the system (\$)], 
\addplot+[samples=400] 
file {DataFigureV3/CopulaDependentOPF_Independent_rho1_01_Exp2.dat}; 
\addlegendentry{$\theta_1$ = 0.1}
\addplot+[samples=400] 
file {DataFigureV3/CopulaDependentOPF_Independent_rho1_008_Exp2.dat};  
\addlegendentry{$\theta_1$ = 0.08}
\addplot+[samples=400] 
file {DataFigureV3/CopulaDependentOPF_Independent_rho1_005_Exp2.dat};  
\addlegendentry{$\theta_1$ = 0.05}

\addplot[name path = B,	opacity = 0] 
	file {DataFigureV3/CopulaDependentOPF_Independent_rho1_01_stdm2.dat};
\addplot[name path = C, opacity = 0] 
	file {DataFigureV3/CopulaDependentOPF_Independent_rho1_01_stdp2.dat};
\addplot[name path = E,	opacity = 0] 
	file {DataFigureV3/CopulaDependentOPF_Independent_rho1_005_stdm2.dat};
\addplot[name path = F, opacity = 0] 
	file {DataFigureV3/CopulaDependentOPF_Independent_rho1_005_stdp2.dat};
\addplot[name path = H,	opacity = 0] 
	file {DataFigureV3/CopulaDependentOPF_Independent_rho1_008_stdm2.dat};
\addplot[name path = I, opacity = 0] 
	file {DataFigureV3/CopulaDependentOPF_Independent_rho1_008_stdp2.dat};

\addplot [blue,fill=blue!90!black,opacity=0.3]
    fill between[of=B and C];
\addplot [red,fill=brown!90!black,opacity=0.3]
    fill between[of=E and F];
\addplot [brown,fill=red!90!black,opacity=0.3]
    fill between[of=H and I];

\end{semilogxaxis} 
\end{tikzpicture}\label{Indtheta1}}} 
 	\caption{Numerical study. The shaded area around each curve represents the corresponding standard deviation. } 
 	\label{NumericalStudy}
 	\vspace{-0.1cm}
\end{figure*}
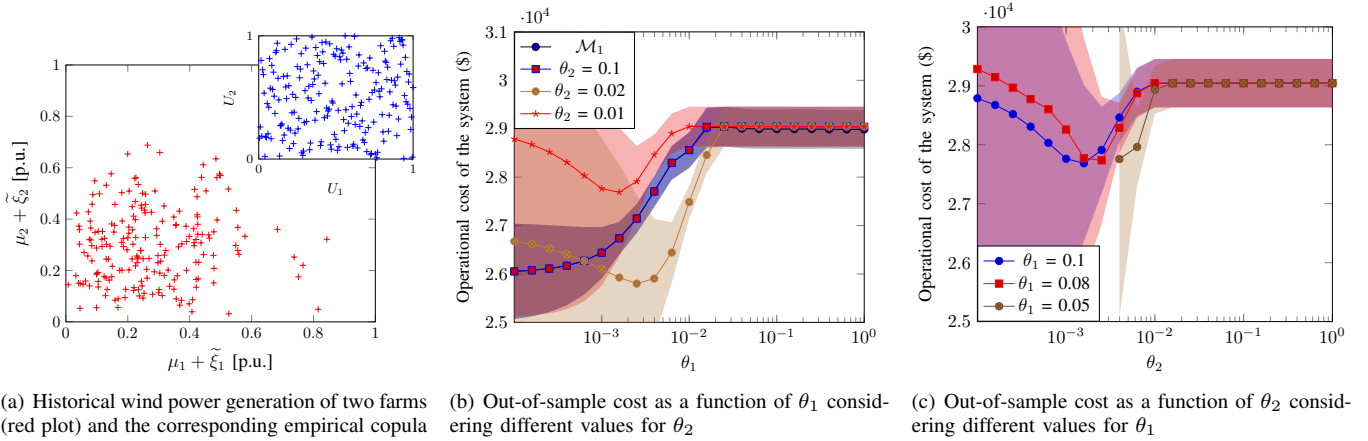

\color{black}




\subsection{Out-of-Sample Performance: Operational Cost of the System as a Function of $\theta_1$}
\label{SectOSA1}
We assign three different values to parameter $\theta_2$, namely $0.01$, $0.02$, and $0.1$. We vary the value of parameter $\theta_1$ from $10^{-4}$ to $10^{0}$, where the exponent increases linearly with a step of $0.2$, and compare the out-of-sample performance of the DRO model with the ambiguity sets $\mathcal{M}_1$ and $\mathcal{M}_2$. We retrieve the average out-of-sample operational cost of the system. In addition, the standard deviation of the cost over $970$ out-of-sample simulations is calculated. These results are illustrated  in logarithmic Fig. \ref{Indtheta2}.


One can observe that the results obtained from the DRO model with the ambiguity set $\mathcal{M}_1$ and those with the ambiguity set $\mathcal{M}_2$ when $\theta_2$ is comparatively high (e.g., $0.1$) are similar. This implies that the second constraint in $\mathcal{M}_2$ regarding the distance with respect to the empirical copula is not binding, and therefore the worst-case distributions for which the optimal decisions are made are identical in both DRO models. In other words, the ambiguity set $\mathcal{M}_2$ is identical to $\mathcal{M}_1$ when $\theta_2$ takes comparatively high values. \color{black} In contrast, when the value of $\theta_2$ decreases, e.g., to $0.02$, the second constraint in $\mathcal{M}_2$ becomes binding, leading to improved results in terms of the out-of-sample cost. This improvement has led to the lowest out-of-sample cost, since  the constraint on dependence structure eliminates the unrealistic distributions from the ambiguity set. This numerical finding suggests that the proposed ambiguity set $\mathcal{M}_2$ outperforms  $\mathcal{M}_1$, provided that appropriate values for parameters  $\theta_1$ and $\theta_2$ are selected (i.e., $\theta_1 = 0.0025$ and $\theta_2 = 0.02$). Given these appropriate values for $\theta_1$ and $\theta_2$, we observe a potential cost saving of 1\% in expectation with respect to the classical ambiguity set $\mathcal{M}_1$ given the value for $\theta_1$ = 10$^{-4}$, which achieves the minimal expected total cost for $\mathcal{M}_1$.\color{black}

It is worth mentioning that the DRO model becomes infeasible, when both parameters $\theta_1$ and $\theta_2$ take very low values, meaning that both constraints in $\mathcal{M}_2$ are highly restrictive. \color{black} \color{black} We have observed infeasibility for values of $\theta_1$ and $\theta_2$ lower than 0.01 in our simulations in Section V-C and for values lower than $\theta_1 = 0.05$ and $\theta_2 = 0.004$ in our simulations in Section V-D. One can intuitively interpret the observation on infeasibility as a case under which the two Wasserstein balls, defined by two constraints in $\mathcal{M}_2$, have no intersection. In other words, there is no distribution within the ambiguity set that satisfies both constraints at the same time.

\begin{figure*}

\subfigure[Comparison of the total upward reserve capacity dispatched (i.e., $\mathds{1}^\top \overline{r}$), obtained from the DRO model under different settings. Fixed values: $\epsilon = 0.05$, $\theta_1 = 0.1$, and $N = 30$.]{\resizebox {0.66\columnwidth} {!} {\begin{tikzpicture}
\begin{axis}[
    ybar,
    enlargelimits=0.5,
    legend style={at={(0.5,-0.15)},
      anchor=north,legend columns=-1},
    ylabel={Total upward reserve capacity [MW]},
    symbolic x coords={A},
    xtick=data,
    xticklabels={},
    nodes near coords,
    nodes near coords align={vertical},
    bar width=25pt,
    ]
\addplot coordinates {(A,266.9)};
\addplot coordinates {(A,174.2)};;
\addplot coordinates {(A,266.9)};;
\legend{$\mathcal{M}_1$,$\mathcal{M}_2$ ($\theta_2 = 0.001$),$\mathcal{M}_2$ ($\theta_2 = 0.1$)}
\end{axis}
\end{tikzpicture} \label{DispatchRU}}} 
 	\enspace
 	\subfigure[Computational time as a function of the number of wind farms. Fixed values: $\theta_1 = 0.1$, $\epsilon = 0.05$, and $N = 15$.]{\resizebox {0.66\columnwidth} {!} {\begin{tikzpicture}
\begin{axis}[
    enlargelimits=false,xmin=1, xmax=11,
    ymin=0, ymax=500,
    xtick={2,4,6,8,10,12},
    ytick={50,100,150,200,250,300,350,400,450,500},
    xlabel = {Number of wind farms},
    ylabel = {Computational time [second]},
    legend columns=1,
    legend style={at={(0,1)},anchor=north west},
  ]
    \addplot+[samples=400,line width = 2pt,color=black] 
    file {DataFigure/CopulaDependent_Independent_scal_01_TIME_N15v3.dat}; 
\addlegendentry{$\theta_2$ = 0.1}
\addplot+[samples=400,line width = 2pt,]
file {DataFigure/CopulaDependent_Independent_scal_002_TIME_N15v3.dat}; 
\addlegendentry{$\theta_2$ = 0.02}
\addplot+[samples=400,line width = 2pt,]
file {DataFigure/CopulaDependent_Independent_scal_001_TIME_N15v32.dat}; 
\addlegendentry{$\theta_2$ = 0.01}
\addplot+[samples=400,line width = 2pt,color=blue]
file {DataFigure/StandardDRO_Independent_scal_001_TIME_N15v3.dat}; 
\addlegendentry{$\mathcal{M}_1$}
 \coordinate (insetPosition) at (rel axis cs:0,0.5);
\end{axis} 
\end{tikzpicture}\label{FigScalabilityTime}}}
\enspace
\subfigure[Out-of-sample operational cost of the system as a function the number of wind farms. Legends are the same at those in plot (b). Fixed values: $\theta_1 = 0.1$, $\epsilon = 0.05$, and $N = 15$.]{\resizebox {0.66\columnwidth} {!} {
\begin{tikzpicture}

\begin{axis}[
      ymin=17500, ymax=21500,
    ytick={18000,19000,20000,21000},
    xtick={2,4,6,8,10,12},
    xlabel = {Number of wind farms},
    ylabel = {Operational cost of the system [\$]},
    enlargelimits=false,
    legend style={at={(0,0)},anchor=north west},
    axis background/.style={fill=white!100},
    ]

    \addplot+[samples=400,line width = 2pt,dashed,color=black] 
file {DataFigure/CopulaDependent_Independent_scal_01_Exp_N15.dat}; 
\addplot+[samples=400,line width = 2pt,dashed] 
file {DataFigure/CopulaDependent_Independent_scal_002_Exp_N15.dat};
\addplot+[samples=400,line width = 2pt,dashed] 
file {DataFigure/CopulaDependent_Independent_scal_001_Exp_N15.dat};
\addplot+[samples=400,line width = 2pt,dashed,color=blue] 
file {DataFigure/StandardDRO_Independent_scal_001_p_N15_Exp.dat};
    
  \end{axis}
 
\end{tikzpicture} \label{FigScalabilityCost}}}
 	\caption{Operational results with two (the first plot) and more (the next two plots) wind farms. The aggregate capacity of  farms is always equal to $1,000$ MW. } 
 	    \vspace{-0.1cm}
\end{figure*}
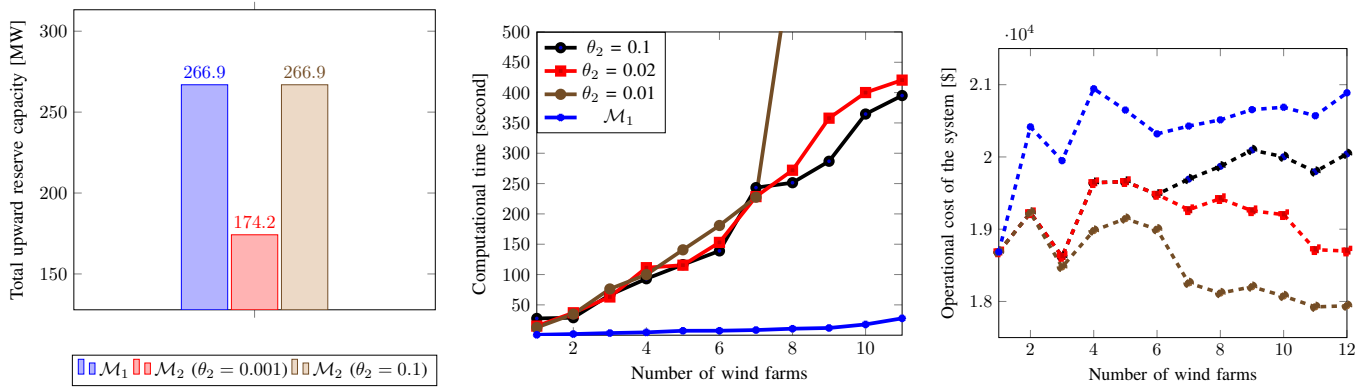

\subsection{Out-of-Sample Performance: Operational Cost of the System as a Function of $\theta_2$}
\label{SectOSA2}
We fix the value of parameter $\theta_1$ to either $0.05$, $0.08$, or $0.1$, and report the out-of-sample operational cost of the system for different values of $\theta_2$. In particular, this value ranges from $10^{-4}$ to $10^{0}$, where the exponent increases linearly with a step equal to $0.2$. Fig. \ref{Indtheta1} shows the out-of-sample results in a logarithmic scale.

\color{black} Our first observation is that for comparatively high values of $\theta_2$, the outcomes of the DRO model with the ambiguity set $\mathcal{M}_2$ are very similar to those with the ambiguity set $\mathcal{M}_1$. This numerically indicates that the copula constraint is not binding. Second, we observe that the cost curves reach  a minimum at intermediate values for $\theta_2$, suggesting that an optimal value for the parameter $\theta_2$ exists. This difference between the outcomes obtained by the DRO model with $\mathcal{M}_1$ and those with $\mathcal{M}_2$, given the optimal value for $\theta_2$, represents the maximum cost earned due to the use of the more elegant $\mathcal{M}_2$ instead of $\mathcal{M}_1$, for fixed value of $\theta_1$. For instance, when $\theta_1 = 0.1$, the optimum is obtained for $\theta_2 = 0.0016$, and it corresponds to a cost saving of 4.7\% compared to $\mathcal{M}_1$ with the same value of $\theta_1 = 0.1$. \color{black} Third, we observe that the program becomes infeasible for comparatively low values of $\theta_2$, when $\theta_1$ is also low (e.g., $0.01$), confirming our intuition in the previous analysis.

\color{black}

\color{black}

\begin{remark}[On the selection of $\theta_1$ and $\theta_2$ by the system operator]
It is worth mentioning that when the decision-maker has less confidence in the available representation of uncertainty and dependence structure (e.g., when the number of historical in-sample observations is relatively low), the system operator should select $\theta_1$ and $\theta_2$ as large values, e.g., $\theta_1 = \theta_2 = 1$. This would result in similar outcomes to those obtained by the state-of-the-art metric-based ambiguity set. Next, by slightly decreasing the value of $\theta_1$ and $\theta_2$ from day-to-day, the system operator should observe a decrease in the average total operating cost, followed by an increase (corresponding to a displacement from right to left on the curves of Fig. 2(b) and 2(c)). The meaning behind this observation is that the optimal values for $\theta_1$ and $\theta_2$ have been found and should no longer be decreased, for the corresponding OPF problem to be solved.
\end{remark}

\color{black}

\subsection{Operational Decisions}

We explore the impacts of the proposed ambiguity set $\mathcal{M}_2$ on the resulting operational decisions. Fig. \ref{DispatchRU}  reports the total upward reserve capacity, i.e., $\mathds{1}^\top \overline{r}$, procured from conventional units in the day-ahead stage. In particular, we solve the DRO model under three different settings: (\textit{i}) with the ambiguity set $\mathcal{M}_1$, (\textit{ii}) with the ambiguity set $\mathcal{M}_2$ where $\theta_2=0.001$, and (\textit{iii}) with the ambiguity set $\mathcal{M}_2$ where $\theta_2=0.1$. In all three cases, we fix the value of $\theta_1$ to be $0.1$. 
We observe that when $\theta_2$ takes a comparatively higher value (i.e., $0.1$), the resulting operational decisions obtained from the DRO model with the ambiguity set $\mathcal{M}_2$ are identical to those with the ambiguity set $\mathcal{M}_1$. This confirms our earlier observation that the copula constraint in $\mathcal{M}_2$ is non-binding. 
On the contrary, when the parameter $\theta_2$ is given an appropriate value (here, e.g., $0.001$), not only the out-of-sample cost (see Fig. \ref{Indtheta1}) but also the total amount of reserve capacity (see Fig. \ref{DispatchRU}) decreases. This interesting observation implies that by elegantly incorporating the dependence structure into the ambiguity set definition, the DRO model will provide \textit{less} conservative operational decisions.

\color{black}
\subsection{Impacts of Increasing the Number of Wind Farms}
\label{Sect:Scalability}

We increase the number of wind farms connected to the system, enlarging the dimension of the uncertainty space. We consider up to 12 wind farms, which represent the number of wind uncertainty sources which is typically considered for studies related to the real-life Belgian electricity grid (i.e., one source per province and an additional one for off-shore wind). 
We re-scale the wind farm capacities, such that the aggregate capacity of wind farms is always equal to $1,000$ MW. We consider $15$ historical observations. 


Fig. \ref{FigScalabilityTime} shows the computational time  to solve the problem \eqref{DR} with both ambiguity sets $\mathcal{M}_1$ (metric-based) and $\mathcal{M}_2$ (with different values of $\theta_2$). The value of parameter $\theta_1$ is fixed to $0.1$ in all cases. In most cases, the computational time increases linearly with the number of wind farms. However, this increase is drastic when the value of $\theta_2$ is comparatively low, indicating the case with a tight copula constraint. 

\color{black} Fig. \ref{FigScalabilityCost} illustrates the out-of-sample operational cost of the system as a function of the number of wind farms. Legends are the same as those in Fig. \ref{FigScalabilityTime}. There are two important observations. First, we obtain lower values for the operational cost under the ambiguity set $\mathcal{M}_2$, in particular when the value of $\theta_2$ is comparatively low, e.g., $0.01$ or $0.02$, meaning that the copula constraint is binding and that the dependence structure plays a role on the collection of distributions within $\mathcal{M}_2$. Second, the operational cost decreases by increasing the number of wind farms, provided that the value of $\theta_2$ is comparatively low. Therefore, the benefits of using $\mathcal{M}_2$ instead of $\mathcal{M}_1$ increases when the number of wind farms is comparatively high. Both observations highlight the importance of adding dependence structure in the definition of the ambiguity set, particularly when the number of wind farms is relatively high. \color{black}

\color{black}
\subsection{Performance in Case of Radial Distribution Systems}
\label{SubSect:Distri}

We consider a radial distribution system composed of 15 nodes, two of which host a controllable generator, each being able to produce electricity for up to 1 MW. Two wind turbines with a total capacity of 1 MW are also connected to the distribution system. The complete dataset of economical and technical parameters is provided in the online companion \cite{OnlineAppendix}.

The renewable and dependent wind power generation introduces uncertainty in the operation of the system. To cope with this uncertainty, we solve the distributionally robust OPF \eqref{DRACOPF} for radial distribution systems, given the metric-based ambiguity set $\mathcal{M}_1$ and the proposed copula-based ambiguity set $\mathcal{M}_2$. We leverage the per-unit wind power deviation dataset generated in Section \ref{SubSect:OSA} for the meshed transmission case study and perform an out-of-sample analysis, with a procedure similar to the one showcased in Section \ref{SubSect:OSA}, to fairly compare different models with respect to unseen realizations of uncertainty. The formulation of the underlying deterministic real-time optimization problem is available in the online companion \cite{OnlineAppendix}. 

\begin{table}[]
\caption{\color{black} Numerical study for radial distribution systems given metric-based and copula-based approaches. The cells in blue highlight the best results obtained in terms of total operating cost. Fixed values: $N$ = 30 and $\epsilon$ = 0.05.}
\vspace{0.2cm}
\label{Table:ACOPFResults}
\centering
{\resizebox {\columnwidth} {!} {
\begin{tabular}{ll|lll}
\hline
                                  &                    & \begin{tabular}{@{}l@{}} Expected\\ cost [k\euro]\end{tabular} & \begin{tabular}{@{}l@{}} Standard \\ deviation [k\euro] \end{tabular}  & EENS [MWh]  \\ \hline
\multirow{3}{*}{$\theta_1 = 0.1$}    &  Metric-based                &  120.76             &   11.7                &   0.068                                                 \\ 
 &  $\theta_2 = 0.01$  &  120.9             &  11.7                  &  0.068                                             \\
                                  & $\theta_2 = 0.001$ &  109.1             &  11.7                  &  0.069                                                   \\ \hline
\multirow{3}{*}{$\theta_1 = 0.05$}    & Metric-based & 120.8                  & 11.7              & 0.068                                    \\ 
                & $\theta_2 = 0.001$   &  109.1             & 11.7                   &   0.069                                                \\
                                  & $\theta_2 = 0.0001$  &  107.5             &  11.6                  & 0.070                                                    \\ \hline
\multirow{3}{*}{$\theta_1 = 0.01$}   & Metric-based & 120.8                  & 11.7              & 0.068                                                     \\ 
                 & $\theta_2 = 0.1$   &  120.8             & 11.7                   &    0.068                                                 \\
                                  & $\theta_2 = 0.01$  &  \cellcolor{blue!25} 106.2             &  \cellcolor{blue!25} 11.6                  & \cellcolor{blue!25} 0.072                                                       \\ \hline
\end{tabular}}}
\end{table}

We report the results in terms of expected total operating cost and its standard deviation, as well as the Expected Energy Not Served (EENS) in MWh, in Table \ref{Table:ACOPFResults}. The EENS is an indicator that is usually chosen by the system operator to evaluate the amount of load that is expected to be shed during the corresponding time period. In addition to providing insights of constraint violation probability, this indicator embeds the information of the severity of the constraint violation.

Similarly to the results in Sections VI.C and VI.D, we observe that for given appropriate values for $\theta_1$ and $\theta_2$, e.g., $\theta_1 = 0.01$ and $\theta_2 = 0.01$, the proposed copula-based approach achieves the lowest expected total operating cost (and standard deviation) compared to the traditional metric-based approach. For a fixed value of $\theta_1$, the maximum cost saving observed in our numerical experiments may reach up to 12 \% compared to the traditional metric-based approach with the same fixed $\theta_1$, e.g., when $\theta_1 = 0.01$. Regarding the EENS, we observe a decrease of curtailed load when the parameters $\theta_1$ and $\theta_2$ increase. Recall that the Wasserstein radii $\theta_1$ and $\theta_2$ relate to the distributional robustness. This means that, when $\theta_1$ or $\theta_2$ increase, the number of distributions within the set increases, ensuing a potentially more conservative worst-case distribution. This, in turn, impacts the empirical violation probabilities (and therefore, the EENS), because the violation probability $\epsilon$ has not the same implication when evaluated for different distributions.


\color{black}

\color{black}

%


\section{Conclusion}
\label{Sect:Conclusion}

This paper introduces an ambiguity set that includes an additional Wasserstein constraint on copula, therefore capturing the whole dependence structure among all uncertain parameters. We develop a generic distributionally robust model that can be applied to any kind of decision-making optimization problems in power systems under uncertainty. In particular, we apply the proposed model to a distributionally robust day-ahead OPF problem. The results show the potential for a significant operational cost saving for the whole system, achieved by taking into account the dependence structure of the uncertain renewable energy sources. 

\color{black} As a potential path for future research, the development of a tool, e.g., using machine learning techniques, that determines \textit{a priori} the optimal values of parameters $\theta_1$ and $\theta_2$ may help the decision-maker efficiently use the proposed model. Finally, it is interesting to apply the outcome of Theorem \ref{TheoremCopula} to various short-term operational and long-term planning decision-making problems under uncertainty in power systems to further illustrate the potential benefits of the proposed model. \color{black} In particular, we highlight the extension to AC-OPF formulation for the meshed transmission networks. \color{black}

\appendices
\section{Proof of Lemma \ref{Lemma}}
\label{ProofLemmaAppendix}

\begin{proof}
\color{black} For a given argument value $\eta$, the program counts the number of historical observations $\widehat{\xi}_{ki}, \enspace  i \in \left\lbrace 1, ..., N \right\rbrace$ corresponding to the renewable power unit $k$ under which the renewable power generation is lower than $\eta$. The constraint \eqref{Conspos} imposes that the variable $z_{ki}$ takes a non-negative value whenever $\eta$ is higher than the underlying observation $\widehat{\xi}_{ki}$, and vice versa. However, the value of $z_{ki}$ is restricted to lie within 0 and 1 in \eqref{Consbound}. Therefore, the optimal value of $z_{ki}$ will mimic the  function $\mathbbm{1}_{\eta \geq \widehat{\xi}_{ki}}$. This enables the  objective function \eqref{LemmaOF} to compute the value of the empirical marginal cumulative distribution function. \color{black}
\end{proof}

\section{Proof of Theorem \ref{TheoremCopula}}
\label{ProofTheoremAppendix}

\begin{proof}
We depart from the equivalent reformulation of the worst-case expectation problem in \eqref{DRWCE}, given by
\begingroup
\small
\allowdisplaybreaks
\begin{subequations}\label{Proof1}
\begin{align}
& \min_{\alpha, \beta \geq 0, y_i} \alpha \theta_1 + \beta \theta_2 + \frac{1}{N} \sum\limits_{i = 1}^N y_i \\
& \text{ s.t.} \enspace y_i \geq \max_{\xi \in \Xi} \enspace a \left( x \right)^\top \xi + b \left( x \right) - \alpha \, \text{d} \left( \widehat{\xi}_i, \xi \right) - \beta \, \text{d}_F \left( \widehat{\xi}_i, \xi \right) \enspace \forall i. \label{Constraint1}
\end{align}
\end{subequations}
\normalsize
\endgroup

Note that \eqref{Proof1} is equivalent to the findings in \cite[Theorem 2]{Gao2017Distributionally}. For the sake of completeness, we also provide the proof of assertion \eqref{Proof1} in the online companion \cite{OnlineAppendix}. Complicating constraint \eqref{Constraint1} requires reformulations, as it contains a maximization operator over variable $\xi$ and the distance functions, which are defined in the following. \color{black} From now on, the proof will focus on the reformulation of \eqref{Constraint1}. The distance functions are defined using norms, such that 

\begingroup
\small
\allowdisplaybreaks
\begin{subequations}
\begin{align}
d \left( \widehat{\xi}_i, \widetilde{\xi} \right) = \norm{ \widehat{\xi}_i - \widetilde{\xi}} \text{ and } \text{d}_F \left( \widehat{\xi}_i, \widetilde{\xi} \right) = \norm{ F \left( \widehat{\xi}_i \right) - F \left( \widetilde{\xi} \right) }
\end{align}
\end{subequations}
\normalsize
\endgroup
%
%
\noindent \color{black} where $F(\widehat{\xi}_i) =  \left( F_1 \left( \widehat{\xi}_{1i} \right), ..., F_k \left( \widehat{\xi}_{ki} \right), ..., F_{\vert \mathcal{W} \vert} \left( \widehat{\xi}_{\vert \mathcal{W} \vert i} \right) \right)^\top$ and $F(\widetilde{\xi}) =  \left( F_1 \left( \widetilde{\xi}_{1} \right), ..., F_k \left( \widetilde{\xi}_{k} \right), ..., F_{\vert \mathcal{W} \vert} \left( \widetilde{\xi}_{\vert \mathcal{W} \vert} \right) \right)^\top$ are vectors in $\mathbb{R}^{\vert \mathcal{W} \vert}$. In other words, each component of $\widehat{\xi}_i$ or $\widetilde{\xi}$ is given as argument of its corresponding marginal cumulative distribution function. Note that the resulting vector $F\left( \widehat{\xi}_i \right)$ is a sample of the copula, which can be evaluated \textit{a priori} using \eqref{CopulaEq}. Note also that $F \left( \widetilde{\xi} \right)$ is a decision variable related to the variations within the variable copula, which requires further reformulations. \color{black} Using such distance definitions, \eqref{Constraint1} can be recast into

\begingroup
\small
\allowdisplaybreaks
\begin{align} y_i \geq \max_{\widetilde{\xi} \in \Xi} \enspace a \left( x \right)^\top \widetilde{\xi} + b \left( x \right) - \alpha \norm{ \widehat{\xi}_i - \widetilde{\xi} } - \beta \norm{ F \left( \widehat{\xi}_i \right) - F \left( \widetilde{\xi} \right) }.
\end{align}
\normalsize
\endgroup

To get rid of the norms inside the objective function of the inner maximization problem, we use dual norms ($ \norm{ x }_\ast = \max_{\norm{v} \leq 1} v^\top x $) as 

\begingroup
\small
\allowdisplaybreaks
\begin{align}
y_i \geq & \max_{\widetilde{\xi} \in \Xi} a \left( x \right)^\top \widetilde{\xi} + b \left( x \right)  - \alpha  \max_{ \norm{\zeta_{i}^{(1)} }_\ast \leq 1} {\zeta_{i}^{(1)}}^\top \left( \widehat{\xi}_i - \widetilde{\xi} \right) \nonumber \\ & - \beta \max_{ \norm{\zeta_{i}^{(2)} }_\ast \leq 1} {\zeta_{i}^{(2)}}^\top \left( F \left( \widehat{\xi}_i \right) - F \left( \widetilde{\xi} \right) \right) \enspace \forall i.  \label{Proof2}
\end{align}
\normalsize
\endgroup

Next, we eliminate the maximization operators on variables $\zeta_{i}^{(1)}$ and $\zeta_{i}^{(2)}$ by (\textit{i}) switching the $- \max$ to a $\min -$, (\textit{ii}) moving the resulting minimization operators to the left, (\textit{iii}) merging the min operators over variables $\zeta_{i}^{(1)}$ and $\zeta_{i}^{(2)}$, \color{black} (\textit{iv}) permuting with the max operator over $\widetilde{\xi}$ (which is allowed using the reformulation given by Lemma 1, because the objective function is linear and the feasible sets are convex and independent), and finally (\textit{v}) dropping the min operator from the right-hand side of the $\geq$ constraint. The variables $\zeta_{i}^{(1)}$ and $\zeta_{i}^{(2)}$ are added to the overall set of decision variables and the constraints are added to the overall set of constraints, such that


\begingroup
\small
\allowdisplaybreaks
\begin{subequations}
\begin{align}
& \begin{aligned} y_i \geq & \max_{\widetilde{\xi} \in \Xi} \enspace a \left( x \right)^\top \widetilde{\xi} + b \left( x \right) - \alpha {\zeta_{i}^{(1)}}^\top \left( \widehat{\xi}_i - \widetilde{\xi} \right) \\ & - \beta {\zeta_{i}^{(2)}}^\top \left( F \left( \widehat{\xi}_i \right) - F \left( \widetilde{\xi} \right)  \right) \enspace \forall i  \end{aligned} \\
& \norm{ \zeta_{i}^{(1)} }_\ast \leq 1 \enspace \forall i \\
& \norm{ \zeta_{i}^{(2)} }_\ast \leq 1 \enspace \forall i,
\end{align}
\end{subequations}
\normalsize
\endgroup
\noindent is equivalent to \eqref{Constraint1}. With the changes of variables $\alpha \zeta_{i}^{(1)} \rightarrow \zeta_{i}^{(1)}$ and $\beta \zeta_{i}^{(2)} \rightarrow \zeta_{i}^{(2)}$, \color{black} the constraints become
\begingroup
\small
\allowdisplaybreaks
\begin{subequations}
\begin{align}
y_i \geq & \max_{\xi \in \Xi} \enspace a \left( x \right)^\top \xi + b \left( x \right) - {\zeta_{i}^{(1)}}^\top \left( \xi_i - \xi \right) \nonumber \\ 
& - {\zeta_{i}^{(2)}}^\top \left( F \left( \xi_i \right) - F \left( \xi \right)  \right) \enspace \forall i  \\
& \norm{ \zeta_{i}^{(1)} }_\ast \leq \alpha \enspace \forall i  \label{alphaconstraint1} \\
& \norm{ \zeta_{i}^{(2)} }_\ast \leq \beta \enspace \forall i. \label{betaconstraint1}
\end{align}
\end{subequations}
\normalsize
\endgroup

We use Lemma \ref{Lemma} to reformulate the functions $F \left( \xi \right)$. This key step allows us to make the link between the variable distribution and the variable copula, by using the empirical marginal cumulative distributions (see Remark 1). The vector $F \left( \xi \right)$ now becomes \eqref{OptiVector} which completes the proof. \end{proof}

\color{black}

\color{black}
\section{Solution Approach}
\label{SolutionApproachAppendix}


We aim to reformulate problem \eqref{FinalTheorem}, especially constraint \eqref{eq1}. The first step of the reformulation consists in getting rid of the maximization operators within $F \left( \xi \right)$. In that direction, we derive the optimality conditions related to the optimization problems in \eqref{OptiVector} and add them into the constraints of the outer maximization problem. We observe that the strong duality theorem holds for the underlying optimization problems. Therefore, the necessary and sufficient optimality conditions for the $k$-th element of vector $F \left( \xi \right)$ in \eqref{OptiVector} are composed of \textit{(i)} the primal constraints, \textit{(ii)} the dual constraints, and \textit{(iii)} the strong duality equality. This mathematically translates to the following set of constraints:

\begingroup
\allowdisplaybreaks
\small
\begin{subequations}
\begin{align}
    & z_{jki} \left( \xi_k - \widehat{\xi}_{kj} \right) \geq 0 \enspace \forall j \\
    & 0 \leq z_{jki} \leq 1 \enspace \forall j \\
    & \frac{1}{N} + \sigma_{kji} \left( \xi_k - \widehat{\xi}_{kj} \right) - \pi_{kji} \leq 0 \enspace \forall j \\
    & \sigma_{kji}, \pi_{kji} \geq 0 \enspace \forall j \\
    & \frac{1}{N} \sum\limits_{j=1}^{N} z_{kji} = \sum\limits_{j=1}^{N} \pi_{kji}.
\end{align}
\end{subequations}
\normalsize
\endgroup

By doing so, the constraint \eqref{eq1} becomes 

\begingroup
\small
\allowdisplaybreaks
\begin{subequations}\label{bilinear}
\begin{align}
& y_i \geq \left\lbrace \begin{aligned} & \max_{\xi, z_{kji}, \sigma_{kji}, \pi_{kji}} \enspace a \left( x \right)^\top \xi + b \left( x \right) - {\zeta_{i}^{(1)}}^\top \left( \widehat{\xi}_i - \xi \right) \\ & - {\zeta_{i}^{(2)}}^\top \left( F \left( \widehat{\xi}_i \right) - \frac{1}{N} \left( \begin{aligned} & \sum\limits_{j=1}^N z_{j1i} \\ & \hphantom{--} \vdots \\ & \sum\limits_{j=1}^N z_{j\vert \mathcal{W} \vert i}\end{aligned} \right) \right)  \\
& \text{ s.t.} \enspace C \xi \leq D \\
& \hphantom{\text{ s.t.}} \enspace z_{jki} \left( \xi_k - \widehat{\xi}_{kj} \right) \geq 0 \enspace \forall k,j \label{BilinearConstraint}\\
& \hphantom{\text{ s.t.}} \enspace 0 \leq z_{jki} \leq 1 \enspace \forall k,j \\
& \hphantom{\text{ s.t.}} \enspace \frac{1}{N} + \sigma_{kji} \left( \xi_k - \widehat{\xi}_{kj} \right) - \pi_{kji} \leq 0 \enspace \forall k,j \\
    & \hphantom{\text{ s.t.}} \enspace \sigma_{kji}, \pi_{kji} \geq 0 \enspace \forall k,j \\
    & \hphantom{\text{ s.t.}} \enspace \frac{1}{N} \sum\limits_{j=1}^{N} z_{kji} = \sum\limits_{j=1}^{N} \pi_{kji} \enspace \forall k
\end{aligned} \right\rbrace \enspace \forall i,
\end{align}
\end{subequations}
\normalsize
\endgroup

\noindent where the inner maximization operator has been dropped and the optimality conditions have been added to the set of constraints of the outer maximization problem. We write the support as $\left\lbrace \widetilde{\xi} \in \mathbb{R}^{\vert \mathcal{W} \vert} \, \middle\vert \, C \widetilde{\xi} \leq D \right\rbrace$, where $C \in \mathbb{R}^{2\vert \mathcal{W} \vert \times \vert \mathcal{W} \vert }$ and $D \in \mathbb{R}^{\vert \mathcal{W} \vert}$. We notice that the resulting outer maximization problem contains bilinear terms in the form of $z_{jki} \xi_k$ and $\sigma_{kji} \xi_k $. We use the McCormick relaxation of bilinear terms to restore linearity, such that the constraint can be cast into

\begingroup
\small
\allowdisplaybreaks
\begin{subequations}\label{bilinearNew}
\begin{align}
& y_i \geq \left\lbrace \begin{aligned} & \max_{\xi, z_{kji}, \sigma_{kji}, \pi_{kji}} \enspace a \left( x \right)^\top \xi + b \left( x \right) - {\zeta_{i}^{(1)}}^\top \left( \widehat{\xi}_i - \xi \right) \\ & - {\zeta_{i}^{(2)}}^\top \left( F \left( \widehat{\xi}_i \right) - \frac{1}{N} \left( \begin{aligned} & \sum\limits_{j=1}^N z_{j1i} \\ & \hphantom{--} \vdots \\ & \sum\limits_{j=1}^N z_{j\vert \mathcal{W} \vert i}\end{aligned} \right) \right)  \\
& \text{ s.t.} \enspace C \xi \leq D \\
& \hphantom{\text{ s.t.}} \enspace t_{jki} - z_{jki} \widehat{\xi}_{kj} \geq 0 \enspace \forall k,j \label{BilinearConstraint}\\
& \hphantom{\text{ s.t.}} \enspace t_{jki} \geq z_{jki} \xi_k^{\text{min}} \enspace \forall k,j \\
& \hphantom{\text{ s.t.}} \enspace t_{jki} \geq \xi_k + z_{kji} \xi_k^{\text{max}} - \xi_k^{\text{max}} \enspace \forall k,j \\
& \hphantom{\text{ s.t.}} \enspace t_{jki} \leq \xi_k + z_{kji} \xi_k^{\text{min}} - \xi_k^{\text{min}} \enspace \forall k,j \\
& \hphantom{\text{ s.t.}} \enspace t_{jki} \geq z_{jki} \xi_k^{\text{max}} \enspace \forall k,j \\
& \hphantom{\text{ s.t.}} \enspace 0 \leq z_{jki} \leq 1 \enspace \forall k,j \\
& \hphantom{\text{ s.t.}} \enspace \frac{1}{N} + v_{jki} - \sigma_{kji} \widehat{\xi}_{kj} - \pi_{kji} \leq 0 \enspace \forall k,j \\
& \hphantom{\text{ s.t.}} \enspace v_{jki} \geq \sigma_{kji} \xi_k^{\text{min}} \enspace \forall k,j \\
& \hphantom{\text{ s.t.}} \enspace v_{jki} \geq \overline{V} \xi_k + \sigma_{kji} \xi_k^{\text{max}} - \overline{V} \xi_k^{\text{max}} \enspace \forall k,j \\
& \hphantom{\text{ s.t.}} \enspace v_{jki} \leq \overline{V} \xi_k + \sigma_{kji} \xi_k^{\text{min}} - \overline{V} \xi_k^{\text{min}} \enspace \forall k,j \\
& \hphantom{\text{ s.t.}} \enspace v_{jki} \geq \sigma_{kji} \xi_k^{\text{max}} \enspace \forall k,j \\
    & \hphantom{\text{ s.t.}} \enspace \sigma_{kji}, \pi_{kji} \geq 0 \enspace \forall k,j \\
    & \hphantom{\text{ s.t.}} \enspace \frac{1}{N} \sum\limits_{j=1}^{N} z_{kji} = \sum\limits_{j=1}^{N} \pi_{kji} \enspace \forall k
\end{aligned} \right\rbrace \enspace \forall i.
\end{align}
\end{subequations}
\normalsize
\endgroup

Note that the vectors $ \widetilde{\xi}^{\text{max}} \in \mathbb{R}^{\vert \mathcal{W}\vert}$ and $ \widetilde{\xi}^{\text{min}} \in \mathbb{R}^{\vert \mathcal{W}\vert}$ correspond to the maximum and minimum thresholds for the uncertain parameters $\widetilde{\xi}$. In the case of renewable energy sources, $\widetilde{\xi}^{\text{min}}$ and $\widetilde{\xi}^{\text{max}}$ are related to zero and the installed capacity, respectively. We now focus on getting rid of the outer maximization operator. The final step of the reformulation is to dualize \eqref{bilinearNew}, such that the problem becomes a minimization and the operator can be dropped from the constraint. The following problem is a tractable reformulation of \eqref{FinalTheorem}: \vspace{-0.35cm}

\begingroup
\small
\allowdisplaybreaks
\begin{subequations}\label{final2}
\begin{align}
& \min_{\Pi} \enspace \alpha \theta_1 + \beta \theta_2 + \frac{1}{N} \sum\limits_{i = 1}^N y_i \label{TheoremOF2}\\
& \text{s.t.} \enspace y_i \geq b \left( x \right) - {\zeta_{i}^{(1)}}^\top \widehat{\xi}_i - {\zeta_{i}^{(2)}}^\top F \left( \widehat{\xi}_i \right) + {\mu^{(0)}}^\top d \notag\\ 
& \hphantom{\text{---------}} + \sum\limits_{k=1}^{\vert \mathcal{W} \vert} \sum\limits_{j=1}^{N} \left( \mu_{kji}^{(3)} \widetilde{\xi}_k^\text{max} - \mu_{kji}^{(4)} \widetilde{\xi}_k^\text{min} + \mu_{kji}^{(6)} + \frac{1}{N} \mu_{kji}^{(7)} + \right. \notag \\
& \hphantom{\text{---------------------------}} \left. \mu_{kji}^{(9)} \overline{V} \xi_k^{\text{max}} - \mu_{kji}^{(10)} \overline{V} \xi_k^{\text{min}}  \right) \enspace \forall i \numberthis \\
& \enspace a_k\left( x \right) + \zeta_{ik}^{(1)} - C_k^\top \mu^{(0)} \notag\\ 
& \hphantom{\text{---------}} +  \sum\limits_{j=1}^{N} \left( \mu_{kji}^{(3)} + \mu_{kji}^{(4)} + \overline{V} \mu_{kji}^{(10)} - \overline{V} \mu_{kji}^{(9)} \right) = 0 \enspace \forall k, i \\ 
& \enspace  \frac{1}{N} \zeta_{ik}^{(2)} - \mu_{kji}^{(1)} \widehat{\xi}_{kj} - \mu_{kji}^{(2)} \widetilde{\xi}_k^{\text{min}} - \mu_{kji}^{(3)} \widetilde{\xi}_k^{\text{max}} \notag \\ 
& \hphantom{\text{---------}} + \mu_{kji}^{(4)} \widetilde{\xi}_k^{\text{min}} + \mu_{kji}^{(5)} \widetilde{\xi}_k^{\text{max}} - \mu_{kji}^{(6)} + \frac{1}{N} \lambda_k \leq 0 \enspace \forall k,j,i \numberthis \\
& \enspace \mu_{kji}^{(1)} + \mu_{kji}^{(2)} + \mu_{kji}^{(3)} - \mu_{kji}^{(4)} - \mu_{kji}^{(5)} = 0 \enspace \forall k, j, i \\ 
& \enspace - \mu_{kji}^{(7)} + \mu_{kji}^{(8)} + \mu_{kji}^{(9)} - \mu_{kji}^{(10)} - \mu_{kji}^{(11)} = 0 \enspace \forall k, j, i \\ 
& \enspace \mu_{kji}^{(7)} -\lambda_k = 0 \enspace \forall k, j, i \\ 
& \enspace \mu_{kji}^{(7)} - \mu_{kji}^{(8)} \xi_k^\text{min} - \mu_{kji}^{(9)} \xi_k^\text{max} + \mu_{kji}^{(10)} \xi_k^\text{min} + \mu_{kji}^{(11)} \xi_k^\text{max} \leq 0 \enspace  \\ 
&  \enspace \norm{ \zeta_{i}^{(1)} }_\ast \leq \alpha \enspace \forall i \\ 
& \enspace \norm{ \zeta_{i}^{(2)} }_\ast \leq \beta \enspace \forall i.
\end{align}
\end{subequations}
\normalsize
\endgroup

\bibliographystyle{IEEEtran}
\bibliography{bare_jrnl}

\end{document}